\documentclass[11pt,a4paper]{amsart}
\usepackage[left=3cm,right=3cm,top=2cm,bottom=2cm, marginparwidth=4cm, marginparsep=0.5cm]{geometry} 
\usepackage{amsmath,amssymb,mathtools,mathdots} 
\usepackage{amsthm}
\usepackage{graphicx}
\usepackage{mathrsfs}
\usepackage{bbm}
\usepackage{fancyhdr}

\usepackage{marginnote} 
\usepackage[utf8]{inputenc}         
\usepackage[T1]{fontenc}            %
\usepackage[final]{microtype} 
\usepackage{url}                    

\usepackage{tikz-cd}                   
\usetikzlibrary{arrows}                 
\usepackage{adjustbox}
\usepackage[]{color} 
\definecolor{azure}{rgb}{0.0, 0.5, 1.0}
\definecolor{awesome}{rgb}{1.0, 0.13, 0.32}
\usepackage{framed} 
\usepackage[hidelinks]{hyperref}
\hypersetup{
  colorlinks   = true, 
  urlcolor     = azure, 
  linkcolor    = azure, 
  citecolor    = awesome 
}
\usepackage[capitalise]{cleveref}

\tikzset{
commutative diagrams/.cd,
arrow style=tikz,
diagrams={>=latex}} 

\usepackage{enumitem} 
\setlist[itemize]{noitemsep, nolistsep}
\setlist[enumerate]{noitemsep, nolistsep}

\newtheorem{thm}{Theorem}[section]
\newtheorem{prop}[thm]{Proposition}

\newtheorem{lm}[thm]{Lemma}
\newtheorem{cor}[thm]{Corollary}

\newtheorem*{thprop}{Proposition}

\newtheorem{thmx}{Theorem}

\newtheorem{corx}[thmx]{Corollary}

\theoremstyle{definition}
\newtheorem{defn}[thm]{Definition}
\newtheorem{ex}[thm]{Example}

\theoremstyle{remark}
\newtheorem{rke}[thm]{Remark}

\newcommand\scalemath[2]{\scalebox{#1}{\mbox{\ensuremath{\displaystyle #2}}}}
\newcommand{\p}{\mathfrak{p}}
\newcommand{\q}{\mathfrak{q}}
\newcommand{\m}{\mathfrak{m}}
\newcommand{\ie}{\emph{i.e.,} }

\newcommand{\K}{\mathbbm{k}}

\newcommand{\C}{\mathbb{C}}

\newcommand{\Z}{\mathbb{Z}}
\newcommand{\N}{\mathbb{N}}
\newcommand{\Pj}{\mathbb{P}}

\newcommand{\Or}{\mathcal{O}}
\newcommand{\Ka}{\mathcal{K}}

\newcommand{\Hom}{\mathrm{Hom}}
\newcommand{\Ext}{\mathrm{Ext}}
\newcommand{\proj}{\mathrm{Proj}}

\newcommand{\Jac}{\mathrm{Jac}}
\newcommand{\cid}{\mathrm{cid}}
\newcommand{\Id}{\mathcal{I}}
\newcommand{\Fitt}{\mathrm{Fitt}}
\newcommand{\Spec}{\mathrm{Spec}}
\newcommand{\A}{\mathbb{A}}
\newcommand{\ord}{\mathrm{ord}}
\title{The complete intersection discrepancy of a curve I: Numerical Invariants}
\author{Andrei Bengu\textcommabelow s-Lasnier and Antoni Rangachev
\vspace{.8cm}\\
W\MakeLowercase{ith an appendix by} Marc Chardin}

\begin{document}

\pagestyle{fancy}
\fancyhf{}
\fancyhead[CE]{\textsc{Andrei Bengu\textcommabelow s-Lasnier and Antoni Rangachev}}
\fancyhead[CO]{\textsc{The complete intersection discrepancy of a curve}}
\fancyhead[LE,RO]{\thepage}


\address{Andrei Bengu\textcommabelow s-Lasnier\\
Institute of Mathematics and Informatics\\
Bulgarian Academy of Sciences\\
Akad. G. Bonchev, Sofia 1113, Bulgaria \\
email:abengus@math.bas.bg}

\address{Antoni Rangachev\\
Institut de Math\'ematiques de Jussieu-Paris Rive Gauche \\
Centre national de la recherche scientifique (CNRS) \\
Paris, France\\
Institute of Mathematics and Informatics\\
Bulgarian Academy of Sciences\\
email: rangachev@imj-prg.fr}

\begin{abstract}
We generalize two classical formulas for complete intersection curves by introducing the {\it complete intersection discrepancy} of a curve as a correction term. The first is a well-known multiplicity formula in singularity theory, due to Lê, Greuel and Teissier, which relates some of the basic invariants of a curve singularity. We apply this generalization elsewhere to the study of equisingularity of curves. The second is the genus--degree formula for projective curves. The main technical tool used to obtain these generalizations is an adjunction-type identity derived from Grothendieck duality theory.

\end{abstract}


\maketitle

\tableofcontents

\section{Introduction}\label{intro}

Throughout this paper $\K$ is an algebraically closed field of arbitrary characteristic unless specified otherwise. All schemes considered
are equidimensional and of finite type over $\K$. A curve is a scheme of dimension one. 
Let $X$ be a Cohen-Macaulay curve and let
$Z$ be a complete intersection curve in some ambient smooth variety. Assume there exists a closed immersion $\iota:X\hookrightarrow Z$ such that $X=Z$ at the generic point of each irreducible component of $X$. When $X$ is an affine or a projective curve, we show below how to construct $Z$ from the equations of $X$. Set
$W:=\overline{Z\setminus X}$. In the language of linkage, we say that $X$ and $W$ are {\it geometrically linked} by the complete intersection $Z$.



We want to compute invariants of $X$ through those of $Z$ by quantifying the "difference" between the two curves as follows.

\begin{defn}
Let $x$ be a closed point in $X$. Denote by $\Id_X$ and $\Id_W$ the ideal sheaves of $X$ and $W$ in
$\Or_Z$, respectively. Define the \textit{complete intersection discrepancy} of $X$ with respect to $Z$ at $x$ as
the intersection number $I_x(X,W):=\dim_{\K}\Or_{Z,x}/(\Id_{X}+\Id_{W}).$ The \textit{complete intersection discrepancy} of $X$ with respect to $Z$ is
\[I(X,W):=\sum_{x\in X\cap W}I_x(X,W). \] 
\end{defn}

When $X$ is projective, Hartshorne's connectedness theorem implies that $X$ and $W$ intersect provided that $W$ is nonempty. So $Z=X$ if and only if $I(X,W)=0$. Computing $I(X,W)$ from the definition above would require finding equations for $W$ which can be hard to do. We will show how to compute $I(X,W)$ directly from the equations of $Z$ and $X$  when $X$ has locally smoothable singularities using a result from \cite{BGR} which shows that $I_x(X,W)$ is constant in flat families. 

Assume $X$ is a reduced curve. Let $x \in X$. Consider the germ $(X,x) \subset \mathbb{A}_{\K}^n$. As is conventional when dealing with germs, we will consider an affine representative of it. As before, let $(Z,x)$ be a complete intersection curve that contains $(X,x)$ and that is equal to $(X,x)$ at the generic point of each irreducible component of $(X,x)$. Using linear algebra and prime avoidance we show in Proposition~\ref{prop:reduced_construction} that such $Z$ can be defined by $n-1$ general $\K$-linear combinations of a set of equations for  $(X,x)$. When $Z$ is general, we will show  below that the complete intersection discrepancy $I_x(X,W)$ is an intrinsic invariant of $X$ which we denote by $\cid(X,x)$. A topological interpretation of $\cid(X,x)$ when $(X,x)$ is smoothable is provided in \cite{PR}.

Denote by $\Jac(X,x)$ and $\Jac(Z,x)$ the Jacobian ideals of $(X,x)$ and $(Z,x)$, respectively. Identify $\Jac(Z,x)$ with its image in $\Or_{X,x}$. As the two ideals are primary to the maximal ideal of $\Or_{X,x}$ they have well-defined Hilbert--Samuel multiplicities that we will denote by $e(\Jac(X,x))$ and $e(\Jac(Z,x))$. Denote by $m_x$ the multiplicity of $X$ at $x$.

Let $\nu:\overline{(X,x)}\to (X,x)$ be the normalization morphism. The number $r_x:=|\nu^{-1}(x)|$ is the \textit{number of branches at} $x$. Define the \textit{delta invariant} of $X$ at $x$ as $\delta_{x}:=
\dim_\K\nu_*\Or_{\overline{X,x}}/\Or_{X,x}$ and the Milnor number of $X$ at $x$ as $\mu_x:=2\delta_x-r_x+1$. Finally, define the \textit{ramification ideal} $R_X:=\Fitt_0
(\Omega^1_{\overline{X}/X})$ of $\overline{X}$ over $X$ as the $0$-th Fitting ideal of the module of relative Kähler differentials. It is an ideal in $\Or_{\overline{X}}$.
Denote by
$e(R_X):=e(\nu_*R_X)$ the Hilbert--Samuel multiplicity of $\nu_*R_X$ in $\mathcal{O}_{X,x}$. The following relation among the invariants introduced above (see Theorem \ref{thm:rho_to_multiplicity}) holds.

\begin{thmx}\label{thm:Milnor}
We have
$$e(\Jac(Z,x))-I_x(X,W)=2\delta_x+e(R_X).$$
When $Z$ is general, we have $e(\Jac(X,x))=e(\Jac(Z,x))$ and $I_x(X,W)=\cid(X,x)$. In addition, if $x$ is a tame point (e.g.\ $\mathrm{char}(\K)=0)$,  then $e(R_X)=m_x-r_x$ and thus
$$e(\Jac(X,x))-\cid(X,x)=\mu_x+m_x-1.$$
\end{thmx}


Theorem \ref{thm:Milnor} was established for plane curves by Teissier
\cite[Proposition II.1.2]{T73}, for complete intersection curves by Lê \cite{Le74}
and Greuel \cite{Greuel73} (see \cite{BMSS16} for an overview and \cite{DG14} for connections between $\mu_x+m_x-1$ and the local Euler obstruction), and for smoothable curves by the two authors and Gaffney
in \cite{BGR}. In \cite{BGR} the authors show that the change in $e(\Jac(X,x))-\cid(X,x)$ across flat families equals the degree of the relative polar variety of smallest dimension, which yields a multiplicity characterization of equisingularity for families of curves. When $(X,x)$ is a Gorenstein curve, in Section \ref{Nash} we interpret $e(\Jac(X,x))-\cid(X,x)$ as the degree of the exceptional divisor of the Nash blowup of $(X,x)$.

Next, we compute the arithmetic genus of a Cohen--Macaulay projective curve 
$X$ that is generically a complete intersection (for instance, when $X$ is reduced),
generalizing the classical genus–degree formula for plane curves.
Suppose $X\subset\Pj_\K^n$ is defined as the zero locus of the homogeneous polynomials
$f_1,\ldots,f_r$. Set $d_i:=\deg(f_i)$ and assume $d_1\geqslant\ldots\geqslant d_r$.
Denote by $(I_X)_{d_i}$ the $d_i$th graded piece of the homogeneous ideal $I_X$ of
$X$ in $\Pj_\K^n$. In Proposition~\ref{prop:generic_CI_construction}
we construct effectively a complete intersection curve $Z\subset\Pj_\K^n$
containing $X$ such that $X$ and $Z$ agree at the generic point of each
irreducible component of $X$. Moreover, $Z$ is defined by homogeneous polynomials
$F_1, \ldots, F_{n-1}$ with $F_i\in (I_X)_{d_i}$ for $i=1, \ldots, n-1$. 


Let us fix some notation. 
Denote by $S(X)$ the homogeneous coordinate ring of $X$. For a linear form $h$ denote by $S(X)_{(h)}$ the degree zero elements in the localization $S(X)_{h}$. Identify the Jacobian ideal $\Jac(Z)$ of $Z$ with its image in  $S(X)_{(h)}$. Denote by $\deg(X)$ the degree of $X$ in $\mathbb{P}_{\K}^n$ and by $r_X$ the number of irreducible components of $X$. Recall that $p_a(X):=1-\chi(X,\mathcal{O}_X)$ is the arithmetic genus of $X$, where $\chi(X,\mathcal{O}_X)$ is the Euler characteristic of the structure sheaf $\mathcal{O}_X$.  When $X$ is smooth and irreducible, the arithmetic genus coincides with the usual genus $g_X:=\dim_{\K}H^{0}(X,\Omega_{X}^1)$.

For a smooth and irreducible complex plane curve $X$ of degree $d$, Clebsch \cite[Chp.\ 20]{PP16} computed $g_X=1+(d-3)d/2$. More generally, if $X=\mathbb{V}(f_1, \ldots, f_{n-1}) \subset \mathbb{P}_{\K}^n$ is a smooth and irreducible complete intersection curve, the adjunction formula gives the well-known genus--degree relation 
\begin{equation}\label{classics}
g_X=1+\frac{(d_1+\cdots+d_{n-1} - n-1)\deg(X)}{2}
\tag{$*$}
\end{equation}
where $\deg (X) = \Pi_{i=1}^{n-1}d_i$. This formula remains valid for arbitrary complete intersection curves if one replaces $g_X$ 
with the arithmetic genus $p_a(X)$.
\begin{thmx}\label{thm:degree}
Suppose $X \subset\Pj_\K^n$ is a Cohen--Macaulay curve that is generically a complete intersection. The following holds
\begin{equation}\label{eq:arithmetic genus}
p_a(X)=1+\frac{(d_1+\cdots+d_{n-1}-n-1)\deg(X)-I(X,W)}{2}.
\end{equation}
Let $\mathbb{V}(h)$ be a hyperplane that avoids $Z_{\mathrm{sing}}$. When $X$ is smooth we have 
\begin{equation}\label{eq:cid-jacobian}
I(X,W)=\dim_{\K} S(X)_{(h)}/\mathrm{Jac}(Z)
\end{equation}
and thus
\begin{equation}\label{eq:genus_X}
g_X=r_X+\frac{(d_1+\cdots+d_{n-1}-n-1)\deg(X)-
\dim_{\K} S(X)_{(h)}/\mathrm{Jac}(Z)}{2}.
\end{equation}
\end{thmx}

The key ingredient in the proofs of Theorem \ref{thm:Milnor} and Theorem \ref{thm:degree} is an observation from Grothendieck duality theory. 
Write $\omega_{X}$ and $\omega_{Z}$ for the dualizing sheaves of $X$ and
$Z$ and $\iota:X\hookrightarrow Z$ for the closed immersion of $X$ into $Z$. We have the following adjunction-type identity (see \cite[Cor.\ 18]{KlRel}, \cite[Prop.\ 9.1 c)]{EM} and \cite[Sect.\ 2]{Ka}) 
\begin{equation}\label{eq:adjunction}
\omega_{X}=\iota^*(\Id_W\omega_{Z})
\end{equation}
which is true more generally assuming that $X$ and $Z$ are equidimensional of the same dimension and $Z$ is Gorenstein. The proof of (\ref{eq:arithmetic genus}) is a direct computation of the Euler characteristics of the sheaves that appear in (\ref{eq:adjunction}). No assumptions on the characteristic of $\K$ are needed. 

The arithmetic genus formula (\ref{eq:arithmetic genus}) is closely related to the well-known linkage formula of Peskine and Szpiro \cite[Proposition 3.1]{PS}, which describes the difference $p_a(X)-p_a(W)$. In Corollary \ref{PS} we recover their result. The Appendix contains an alternative proof of (\ref{eq:arithmetic genus}), due to Marc Chardin, based on resolutions and liaison-theoretic methods; this argument is independent of the proof given in Section  \ref{sec:degree} and does not require that $\K$ be algebraically closed.  In \cite{RT} (\ref{eq:arithmetic genus}) is used to obtain a generalization of the Plücker formula for plane curves to curves of  arbitrary codimension (see Remark \ref{Plücker}).

The significance of (\ref{eq:cid-jacobian}), which is derived in Proposition \ref{jacobian conservation} as a consequence of Theorem \ref{thm:Milnor}, is that to find $I(X,W)$ when $X$ is smooth one does not need to determine the ideal of $W$ in the homogeneous coordinate ring of $Z$, which can be computationally quite involved. 
Moreover,  in Propositions \ref{jacobian conservation} \rm{(ii)} and \ref{smoothable sing} we show that (\ref{eq:cid-jacobian}) generalizes when $X$ is a local complete intersection and, more generally, when $X$ has locally smoothable singularities.  




There is one instance where (\ref{eq:arithmetic genus}) takes a particularly pleasing form.
We say that $X \subset \mathbb{P}_{\K}^n$ is an {\it almost complete intersection} (aci) curve if $X$ can be defined as the zero locus of $n$ homogeneous polynomials $f_1,\ldots,f_n$\footnote{In commutative algebra one requires that the $n$ polynomials form a minimal set of equations for $X$.}.  We say that $X$ is {\it nondegenerate} if $X$ is not contained in a hyperplane of $\mathbb{P}_{\K}^n$.

\begin{corx}\label{genus-degree} Suppose $X \subset\Pj_\K^n$ is a Cohen--Macaulay aci curve that is generically a complete intersection. Then
\begin{equation}\label{aci genus}
p_a(X)=1+\frac{(d_1+\cdots+d_{n}-n-1)\deg(X)-\Pi_{i=1}^n d_i}{2}.
\end{equation}
Suppose $n \geq 3$ and suppose $X$ is a nondegenerate, reduced and connected aci curve. Then 
\begin{equation}\label{degree ineq}
\deg(X) \geq \frac{d_n^n-2}{nd_n-n-1}.
\end{equation}
\end{corx}

It follows from our definition that a complete intersection curve $X \subset \mathbb{P}_{\K}^n$ is an aci curve by repeating any one of its defining equations. With this choice, (\ref{aci genus}) recovers (\ref{classics}).

The paper is organized as follows. In Section \ref{sec:dualizing_sheaf} we introduce the basic notions and results about dualizing sheaves and their affine structure needed for the proofs of \eqref{eq:adjunction} and Theorem \ref{thm:Milnor}. We prove \eqref{eq:adjunction} using Grothendieck's duality theorem for finite morphisms. In Section \ref{sec:Milnor} we use results of Montaldi and van Straten \cite{MVS} about ramification modules and \eqref{eq:adjunction} to prove Theorem \ref{thm:Milnor}. Piene \cite{Pi} defines the $\omega$-Jacobian ideal of a Cohen-Macaulay variety $X$. She shows that when $X$ is Gorenstein, the blowup of $X$ with center the $\omega$-Jacobian ideal gives the Nash modification of $X$. In Section \ref{Nash}, using results from Section \ref{sec:dualizing_sheaf}, we compute Piene's $\omega$-Jacobian \cite{Pi}. In particular, we show that when $X$ is Gorenstein, the $\omega$-Jacobian is equal to the usual Jacobian ideal of $X$ if and only if $X$ is a local complete intersection, thus answering a question of Piene. When $(X,x)$ is a Gorenstein curve, we show that $e(\Jac(X,x))-\cid(X,x)$ is the degree of the exceptional divisor of the Nash blowup of $(X,x)$.

We begin Section \ref{sec:degree} by generalizing an Euler characteristic formula for the degree of a locally free sheaf and then use it along with \eqref{eq:adjunction} to give a proof of Theorem \ref{thm:degree}  (\ref{eq:arithmetic genus}). As a corollary, we show that $I(X,W)$ remains constant under flat deformations of projective curves, and we recover the Peskine--Szpiro formula for the difference of the arithmetic genera of directly linked curves.

In Section \ref{construction} we give an explicit construction of the complete intersection $Z$ using prime avoidance and linear algebra. In Section \ref{computability}, we explain how this construction can be implemented with a computer algebra package. In Proposition \ref{jacobian conservation}, we derive (\ref{eq:cid-jacobian}) from Theorem \ref{thm:Milnor}. We then show how to compute $I(X,W)$ when $X$ is smoothable or when $X$ has locally smoothable singularities. For an aci curve we compute $I(X,W)= \Pi_{i=1}^n d_i-d_n\mathrm{deg}(X)$ which proves (\ref{aci genus}). 
We then show that, under the hypotheses preceding (\ref{degree ineq}), the arithmetic genus $p_a(X)$ is nonnegative. Then (\ref{degree ineq}) follows from (\ref{aci genus}). In Proposition \ref{LGT-aci}, we explain how to compute the complete intersection discrepancy for locally aci curves directly from the equations of $X$ and $Z$.
Finally, we show how to construct a general $Z$ such that when $X$ is a local complete intersection, the points of intersection of  $X$ with $W$ are ordinary double points of $Z$, and their number is equal to $I(X,W)$.

\textbf{Acknowledgments.} We are indebted to Marc Chardin for providing us with helpful comments and bringing to our attention results from linkage theory that we were not aware of. We also thank Terence Gaffney and Bernard Teissier for stimulating conversations. The first author was supported by a ``Peter Beron i NIE" fellowship [KP-06-DB-5] from the Bulgarian Science Fund. The second author was supported by the European Union’s Horizon Europe research and innovation programme under Marie Sk\l{adowska}-Curie Actions project GTSP-$101111114$.

\section{The dualizing sheaf}\label{sec:dualizing_sheaf}

Let us fix some notations and recall some facts about meromorphic functions (see
\cite{KlMisc} and \cite[\href{https://stacks.math.columbia.edu/tag/0EMF}{Tag 0EMF}]{St}).
For any ring $A$, write $K(A)$ for the total
quotient ring of $A$. Suppose $X$ is a reduced scheme of finite type 
over a field $\K$ and of pure dimension $d$. Write $\Ka_X$ for the \textit{sheaf
of meromorphic functions on} $X$. For an open affine subset $U=\mathrm{Spec}(A)\subset X$,
$\Ka_X|_U$ is the module associated to $K(A)$. Let $K=\prod_i K_i$ be the product of the
residue fields $K_i=\kappa(\xi_i)$ of the generic points of the irreducible components of $X$.
Write $j_K:\mathrm{Spec}(K)\to X$ for the canonical map. We have
$\Ka_X=(j_K)_*\Or_{\mathrm{Spec}(K)}$.
We define \textit{the sheaf of meromorphic one-forms} $\Omega^1_{\Ka_X/\K}$ as
\[\Omega^1_{\Ka_X/\K}:=\Omega^1_{X/\K}\otimes_{\Or_X}\Ka_X.\]
Equivalently, $\Omega^1_{\Ka_X/\K}=(j_K)_*\Omega^1_{K/\K}$. Hence it is clear that
$(\Omega^1_{\Ka_X/\K})_x=\prod_{i,x\in\overline{\{\xi_i}\}}\Omega^1_{K_i/\K}$ for each
$x\in X$. Furthermore, if $U=\mathrm{Spec}(A)\subset X$, then $\Omega^1_{\Ka_X/\K}|_U$ is the
module associated to $\Omega^1_{K(A)/\K}$. We will write
$\Omega^k_{X/\K}=\wedge^k\Omega^1_{X/\K}$ and
$\Omega^k_{\Ka_X/\K}=\wedge^k\Omega^1_{\Ka_X/\K}$ for the sheaf of $k$ differential
forms and the sheaf of $k$ meromorphic forms, respectively.



\subsection{Dualizing sheaves}\label{subsec:complementary}
We review some basic facts about the dualizing sheaf of $X$. Our references are \cite[Chapter VII]{Ha}, \cite[Definition (1)]{KlRel} and \cite[Section 9.2]{EM}. To
every equidimensional scheme of finite type $X$ over a field $\K$ of dimension $d$, one
associates a coherent sheaf $\omega_{X}$ with the following properties:
\begin{enumerate}
\item There exists a structural morphism $c_X:\Omega^d_{X/\K}\to\omega_X$, called the
\textit{canonical map}.
\item If $X$ is non-singular, of dimension $d$, then the canonical map is an isomorphism
$\omega_X\simeq\Omega^d_{X/\K}$.
\item The definition is local: if $U\subset X$ is an open subscheme, then there is a
canonical isomorphism $\omega_{U}\simeq\omega_{X}|_{U}$.
\item (Duality theorem for closed embeddings \cite[III, Proposition 7.2, p.\ 179]{Ha})
If $X\hookrightarrow V$ is a closed embedding of pure codimension $e$ and $V$
is equidimensional Cohen-Macaulay scheme, then there is a canonical isomorphism
\[\omega_{X/\K}\simeq\underline{\mathrm{Ext}}^e_{\Or_V}(\Or_X,\omega_{V}).\]
\item (Duality theorem for finite morphisms \cite[III, Proposition 6.7, p.\ 170]{Ha}, and
\cite[Thm. 4.1.1]{BG})
If $p:X\to Y$ is a finite morphism between equidimensional
schemes of finite type with $Y$ Cohen-Macaulay, we have a canonical isomorphism
\[\Bar{p}:p_*\omega_{X/\K}\simeq\underline\Hom_{\Or_{Y}}(p_*\Or_{X},\omega_{Y}).\]
\item $X$ is Gorenstein (\ie every stalk $\Or_{X,x},\,x\in X$ is a Gorenstein local ring)
if and only if $\omega_{X}$ is locally free of rank one.
\end{enumerate}

By facts (3) and (5) we can construct $\omega_{X/\K}$ locally. Indeed for any affine open
$\Spec(A)\subset X$ we can find a Noether normalization $\K[\mathbf{x}]:=\K[x_1,\ldots,x_d]
\subset A$ where $d=\dim(X)$. Thus $\omega_{X/\K}|_U$ is the sheaf associated to the module
$\mathrm{Hom}_{\K[\mathbf{x}]}(A,\K[\mathbf{x}])$.

Observe that for a generically reduced irreducible component $Y$ of $X$, the stalk
$(\omega_{X/\K})_\xi$ at the generic point $\xi$ of $Y$ is $\Omega^1_{\kappa(\xi)/\K}$.

\subsection{The adjunction-type formula}

Let $X$ be a Cohen-Macaulay scheme and
$Z$ be a Gorenstein scheme. Assume $X$ and $Z$ are of the same
dimension and assume there exists a closed immersion $\iota:X\hookrightarrow Z$ such that $X=Z$ at the generic point of each irreducible component of $X$. Denote by $\Id_X$ and $\Id_W$ the ideal sheaves of $X$ and $W$ in
$\Or_Z$, respectively. The ideals of $X$ and $W$ in $\Or_Z$ are related by $\Id_W=((0):_{\Or_Z}\Id_X)$. Write $\omega_{X}$ and $\omega_{Z}$ for the dualizing sheaves of $X$ and
$Z$. Denote by $K_Z$ the canonical class of $Z$ and by $K_X$ the canonical class
of $X$ when $X$ is Gorenstein.

\begin{prop}\label{thm:adjunction}
Suppose that $X$ is Cohen-Macaulay and $Z$ is Gorenstein. We have
$$\omega_{X}=\iota^*(\Id_W\omega_{Z})=\Id_W\Or_X\otimes\omega_Z|_X.$$
Furthermore, if $X$ is Gorenstein, then the image of $\Id_W$ in $\mathcal{O}_X$
defines a Cartier divisor $D_W$ in $X$ and
\begin{equation}
K_X=K_{Z}|_{X} - D_W\label{eq:adjunction-class}
\end{equation}
where $K_X$ and $K_Z$ are the canonical divisor classes of $X$ and $Z$, respectively.
\end{prop}

Proposition \ref{thm:adjunction} appears in the affine setting in the works of  Ein and
Musta\textcommabelow{t}\u{a} in \cite{EM} and Kawakita in \cite{Ka}. It is a well
known result for curves (see
\cite[\href{https://stacks.math.columbia.edu/tag/0E34}{Tag 0E34}]{St}). Here we give
a short proof for any reduced, equidimensional scheme $X$ of finite type over $\K$.

\begin{proof}
By the duality theorem for finite morphisms (see fact (5) in section \ref{subsec:complementary}),
we have a canonical map
\[\Bar{\iota}:\iota_*\omega_{X}\xrightarrow{\simeq}\underline\Hom_{\Or_Z}(\iota_*\Or_{X},\omega_{Z}).\]
Since $\Or_Z\to \iota_*\Or_X$ is surjective, there is an injective morphism
\[\underline\Hom_{\Or_Z}(\iota_*\Or_{X},\omega_{Z})\to
\underline\Hom_{\Or_Z}(\Or_Z,\omega_{Z/\K})=\omega_{Z}\]
that takes $\vartheta$ and sends it to $\vartheta(1)$.

Below we give a short proof of (\ref{eq:adjunction}).
We need to show that there is a canonical isomorphism
\[\underline\Hom_{\Or_Z}(\iota_*\Or_{X},\omega_{Z})\simeq\Id_W\omega_{Z}.\]
Since $\iota$ is an affine morphism and the definitions of $\omega_X$ and $\omega_Z$ are local
(see fact (3) in Section \ref{subsec:complementary}), we can reduce to the affine case. Thus assume
$Z=\mathrm{Spec}(A),X=\mathrm{Spec}(A/I)$ and $W=\mathrm{Spec}(A/J)$. Because $Z$ is
Gorenstein, we can further assume that the affine open
cover we are considering trivializes $\omega_{Z}$. Hence we may suppose that $\omega_{Z}$
is the coherent sheaf associated to $\omega_{A}=A\varpi$, where $\varpi$ is a basis element
of $\omega_{A}$. We have the following natural isomorphisms
$$\Hom_A(A/I,\omega_A) = \Hom_A(A/I,A\varpi)
\simeq\{\alpha\varpi\in A\varpi;\:\alpha I=0\} = ((0):_AI)\varpi=J\omega_A.$$
The second isomorphism identifies $A$-linear maps $\vartheta:A/I\to A\varpi$ with
$\vartheta(1)=\alpha\varpi$ where $\alpha\in A$. This $\alpha$ should satisfy
$\alpha I=0$ since $\alpha I\varpi=I\vartheta(1)=0$. Conversely, for an $\alpha$ such that
$\alpha I=0$, define $\bar{\vartheta}:A\to A\varpi$ that sends $1$ to $\alpha\varpi$.
The condition $\alpha I=0$ is equivalent to $I\subset\mathrm{Ker}(\bar{\vartheta})$. Thus
$\bar{\vartheta}$ factors through the canonical morphism $A\to A/I$. This gives a map
$\vartheta:A/I\to A\varpi$ which sends $1$ to $\alpha$.

The proof of \eqref{eq:adjunction-class} follows from linkage theory. We know that
$\Id_W\Or_X$ is an avatar of $\omega_{X/\K}$ (see \cite[Remarque 1.5]{PS} and
\cite[Thm. 21.23]{Ei}, or Corollary \ref{eq:complete_intersection}). Because $X$ is Gorenstein,
$\Id_W\Or_X$ is locally principal and thus it defines an effective Cartier divisor $D_W$.
The formula for the canonical class is thus a consequence of \eqref{eq:adjunction}.
\end{proof}

\subsection{The affine structure of dualizing sheaves}\label{sec:affine_structure}
In this section we give a detailed proof of parts b) and d) of \cite[Proposition 9.1]{EM}.
We retain the notation of the previous section. We recall some definitions used throughout
the article. Suppose $X,W$ and $Z$
are equidimensional schemes of finite type over $\K$ of dimension $d$, embedded in a nonsingular $n$-dimensional variety $M$ (for instance $M=\A^n_\K$ or $M=\Pj^n_\K$). We say that
\begin{enumerate}
\item $Z$ is a \textit{complete intersection} if the ideal sheaf $\mathcal{I}_Z$
defining $Z$ inside $M$ is generated by $n-d$ global sections.
\item $Z$ is a \textit{local complete intersection (lci) at} $z\in Z$ if the
ideal of $Z$ in $\Or_{M,z}$ can be generated by a regular sequence of
$n-d$ elements.
\item $Z$ is a \textit{generic complete intersection} if $Z$ is an lci at the generic point
of each irreducible component of $Z$.
\item $X$ and $W$ are \textit{geometrically linked via} a complete intersection $Z$ if $X$ and $W$ do not share common components and
$X\cup W=Z$ as schemes. As a consequence, we have
$\mathcal{I}_X=((0):_{\Or_Z}\mathcal{I}_W)$ and
$\mathcal{I}_W=((0):_{\Or_Z}\mathcal{I}_X)$, where $\Id_X$ and $\Id_W$ are the ideal sheaves of $X$ and $W$ in
$\Or_Z$, respectively.
\item $Z$ is \textit{generically reduced along} $X$ if $X$ is a reduced closed subscheme of $Z$ such that $\mathcal{O}_{Z,\eta}=\mathcal{O}_{X,\eta}$ where $\eta$ runs through the generic points of the irreducible components of $X$. 
\end{enumerate}

For the remainder of this section, we assume that $X$ is reduced and that $Z$ is an lci which is generically reduced along $X$. A construction of such a scheme $Z$ is given in \cref{prop:reduced_construction}. We are interested in the local structure of
$\omega_Z$ and $\omega_X$, so we assume that $X$ and $Z$ are affine.
Since $X$ and $Z$ are of finite type over $\K$ we can embed them in some
affine space $A:=\A^n_{\K}$. Suppose $Z$ is defined by an ideal $\mathcal{I}_Z=(F_1,\ldots,F_e)$, where
$e=n-d$. Write $N:=\mathcal{I}_Z/\mathcal{I}_Z^2$ for the normal sheaf of $Z$ in $A$. Observe that $\bigwedge^eN$
is free of rank one and with generator $\Bar{F}=\Bar{F}_1\wedge\cdots\wedge\Bar{F}_e$, where
we write $\Bar{G}$ for the image of $G$ in $N$. Write $dF=dF_1\wedge\cdots\wedge dF_e\in
\Omega^e_{A/\K}$. Let $x_1,\ldots,x_n$  be coordinates on $A$. Since $Z$ is generically reduced along $X$,
and $X$ is Cohen-Macaulay, by \cite[Thm. 18.15.a.]{Ei} the ideal of $e\times e$ minors of the
Jacobian matrix of $(F_1,\ldots,F_e)$ forms an ideal of height at least $1$ in $\Or_{X,\eta}$
and in $\Or_{Z,\eta}$, for each generic point $\eta$ of an irreducible component of $X$.
Up to a general linear change of variables, we can assume that the minor
\[\Delta=\det\left(\frac{\partial F_i}{\partial x_{d+j}}\right)_{1\leqslant i,j\leqslant e}\]
of the Jacobian matrix of $Z$ in $A$ is not identically zero on each irreducible component of
$X$. In other words $\Delta\in\Ka_X^\times$.

\begin{prop}\label{prop:structure}
With the above assumptions the following holds:
\begin{enumerate}
\item[\rm{(1)}] We have
\[\omega_Z=\left(\Omega^n_{A/\K}|_Z\right)\otimes_{\Or_Z}\left(\bigwedge^eN\right)^\vee.\]
\item[\rm{(2)}] The canonical map $c_Z$ is defined by
\[(c_Z)(\alpha_{i_1}\wedge\cdots\wedge\alpha_{i_d})=
(\alpha_{i_1}\wedge\cdots\wedge\alpha_{i_d}\wedge dF)\otimes(\Bar{F}^\vee).\]
The element $\Bar{F}^\vee$ is the linear form $\bigwedge^eN\to\Or_Z$ that takes $\Bar{F}$
to $1$.
\item[\rm{(3)}] The sheaf $\omega_{Z}|_X$ is embedded into $\Omega^d_{\Ka_X/\K}$ and it is
identified with $\Delta^{-1}\Or_Xdx_1\wedge\cdots\wedge dx_d$.
\end{enumerate}
\end{prop}

\begin{proof}
Consider (1). By applying the duality theorem for closed embeddings to $Z\hookrightarrow A$ we
obtain
$\omega_{Z/\K}=\underline{\mathrm{Ext}}^e_{\Or_A}(\Or_Z,\Omega^n_A)$. We can compute this Ext
group via the Koszul complex (see \cite[Ch. III, Proposition 7.2]{Ha}) to obtain
\[\omega_{Z/\K}=\underline\Hom_{\Or_Z}\left(\bigwedge^eN,\Omega^n_{A/\K}|_Z\right)=
\left(\Omega^n_{A/\K}|_Z\right)\otimes_{\Or_Z}\left(\bigwedge^eN\right)^\vee.\]
Consider (2). Recall the conormal sequence
\[\begin{tikzcd}
N\ar[r]&\Omega^1_{A/\K}|_Z\ar[r,"j"]&\Omega^1_{Z/\K}\ar[r]&0.
\end{tikzcd}\]
Write $K:=\mathrm{Ker}(j)$. Since $N$ is free of rank $e$ we have
$\bigwedge^{e+1}K=0$. Thus the morphism
\[\begin{tikzcd}\label{eq:canonical_map}
    \Omega^n_{A/\K}|_Z & \alpha_1\wedge\cdots\wedge\alpha_e\wedge dG_1\wedge\cdots dG_e \\
    \left(\Omega^d_{Z/\K}\right)\otimes\left(\bigwedge^eN\right) \ar[u] &
        (\alpha_1\wedge\cdots\wedge\alpha_e)\otimes(\Bar{G}_1\wedge\cdots\wedge\Bar{G}_e)
        \ar[u,maps to]
\end{tikzcd}\]
is well-defined. Tensoring with $(\bigwedge^eN)^\vee$ and composing with the contraction
$(\wedge^eN)^\vee\otimes(\wedge^eN)\to\Or_Z$, we obtain the desired $c_Z$. This proves (2).

Consider (3). Localize the conormal sequence at $\Ka_X$
\[\begin{tikzcd}
N\otimes\Ka_X \ar[r,"\delta"] &
\Omega^1_{A/\K}\otimes\Ka_X \ar[r] & \Omega^1_{Z/\K}\otimes{\Ka_X}\ar[r] & 0.
\end{tikzcd}\]
Since $Z$ is generically reduced along $X$, by generic smoothness $\mathrm{Ker}(\delta)$ is
torsion. However, $N|_X$ is free, thus any submodule is torsion free. Hence
$\mathrm{Ker}(\delta)=0$. By generic smoothness again $\Omega^1_{Z/\K}\otimes{\Ka_X}$ is free of
rank $d$. The conormal sequence is thus short exact. By
\cite[\href{https://stacks.math.columbia.edu/tag/0FJB}{Tag 0FJB}]{St}
the canonical map $(c_Z)_{\Ka_X}:\Omega^d_{Z/\K}\otimes{\Ka_X}\to\omega_Z\otimes{\Ka_X}$ is an
isomorphism. Because $X$ coincides with $Z$ along the generic points of the irreducible
components of $X$, we can identify $\Omega^d_{Z/\K}\otimes{\Ka_X}$ with $\Omega^d_{\Ka_X}$. Thus
$\omega_Z|_X$ is a submodule of $\Omega^d_{\Ka_X}$.

To find a free generator for $\omega_Z|_X$, we need to find $\alpha=\alpha_1\wedge\cdots
\wedge\alpha_d\in\Omega^d_{\Ka_X}$ such that $(c_Z)_{\Ka_X}(\alpha)=u(dx_1\wedge\cdots
\wedge dx_n)\otimes(\Bar{F}^\vee|_X)$ where $u\in\Or_X^\times$. Clearly, we have
\begin{align*}
(c_Z)_{\Ka_X}(dx_1\wedge\cdots\wedge dx_d)
    &=(dx_1\wedge\cdots\wedge dx_d\wedge dF_1\wedge\cdots\wedge dF_e)\otimes(\Bar{F}^\vee|_X)\\
    &=\Delta(dx_1\wedge\cdots\wedge dx_n)\otimes(\Bar{F}^\vee|_X).
\end{align*}
Thus a basis for $\omega_Z|_X$ is $\Delta^{-1}dx_1\wedge\cdots\wedge dx_d$. This concludes the
proof of (3).
\end{proof}

\begin{rke}\label{rke:affine_structure}
We can illustrate Proposition \ref{prop:structure} with the following diagram
\begin{equation}\label{diag:structure}
\begin{tikzcd}
\omega_X\ar[r,hook,"\Bar{\iota}"] & \omega_Z|_X \ar[r,hook,"\alpha\mapsto\alpha\otimes1"] &[15pt]
    \omega_Z\otimes{\Ka_X} & \\
\Omega^d_{X/\K} \ar[u,"c_X"] \ar[rrr,bend right=20,"\lambda"] &
    \Omega^d_{Z/\K}|_X \ar[l] \ar[r,"\alpha\mapsto\alpha\otimes1"] \ar[u,"c_Z|_X"] &
    \Omega^d_{Z/\K}\otimes{\Ka_X} \ar[u,"(c_Z)_{\Ka_X}"',"\simeq"] &
    \Omega^d_{\Ka_X} \ar[l,equal] \\
\end{tikzcd}
\end{equation}
The map $\Bar{\iota}$ comes from the adjunction-type formula in Proposition \ref{thm:adjunction}. Its
image is thus $\Id_W\omega_Z|_X$. The map $\lambda$ is the localization morphism. For a set
of indices $1\leqslant i_1<\ldots<i_d\leqslant n$, set $D$ to be the minor of the Jacobian of
$(F_1,\ldots,F_e)$ corresponding to the coordinates different from $x_{i_1},\ldots,x_{i_d}$.
Define $\epsilon$ as follows: set $1\leqslant j_1<\ldots<j_e\leqslant n$ the indices different
from $i_1,\ldots,i_d$, then $\epsilon=\mathrm{card}\{(k,l);\,i_k>j_l\}$. Then
\[\lambda(dx_{i_1}\wedge\cdots\wedge dx_{i_d})=
(-1)^\epsilon\frac{D}{\Delta}dx_1\wedge\cdots\wedge dx_d.\]
Indeed,
\begin{align*}
((c_Z)_{\Ka_X}\circ\lambda)(dx_{i_1}\wedge\cdots\wedge dx_{i_d}) &=
(dx_{i_1}\wedge\cdots\wedge dx_{i_d}\wedge dF_1\wedge\cdots\wedge dF_e)
    \otimes(\Bar{F}^\vee|_X)\\
&=(-1)^\epsilon D(dx_1\wedge\cdots\wedge dx_n)\otimes(\Bar{F}^\vee|_X)\\
&=(-1)^\epsilon\frac{D}{\Delta}\Delta(dx_1\wedge\cdots\wedge dx_n)\otimes(\Bar{F}^\vee|_X)\\
&=(-1)^\epsilon\frac{D}{\Delta}(c_Z)_{\Ka_X}(dx_1\wedge\cdots\wedge dx_d).
\end{align*}
Thus
\[\mathrm{Im}(\lambda)=\Delta^{-1}\Jac(Z)\Or_Xdx_1\wedge\cdots\wedge dx_d,\]
where $\Jac(Z)$ is the Jacobian ideal of $Z$.
\end{rke}

\begin{cor}\label{eq:complete_intersection}
Preserve the notations of Proposition \ref{prop:structure} and identify $\omega_X$ with its image in
$\Omega^d_{\Ka_X}$. We have
\[\Delta\omega_X=\Id_W\Or_X dx_1\wedge\cdots\wedge dx_d\]
inside $\Omega^d_{\Ka_X}$
\end{cor}

\textit{Proof.}
By considering the top maps in \eqref{diag:structure} we get
\[\omega_X=\mathrm{Im}(\omega_X\to\Omega^d_{\Ka_X})=
\Id_W\mathrm{Im}(\omega_Z|_X\to\Omega^d_{\Ka_X})=
\Id_W(\Delta^{-1}\Or_Xdx_1\wedge\cdots\wedge dx_d).\qed\]

\section{A multiplicity formula for a curve singularity}\label{sec:Milnor}

Throughout this section $X$ is a reduced equidimensional scheme of finite type over $\K$ of
dimension one. Let $\nu:\overline{X}\to X$ be the normalization morphism. Fix a closed point
$x\in X$.

\subsection{Local ramification invariants}

The number $r_x:=|\nu^{-1}(x)|$ is the \textit{number of branches of $X$ at} $x$.
It is the number of maximal ideals of the integral closure of $\mathcal{O}_{X,x}$ in its total ring of fractions (\cite[\href{https://stacks.math.columbia.edu/tag/0C37}{Tag 0C37}]{St}).
Define the \textit{delta invariant} of $X$ at $x$ as $\delta_x:=
\dim_\K(\nu_*\Or_{\overline{X}})_x/\Or_{X,x}$. Each stalk $\Or_{\overline{X},y},
y\in\overline{X}$ is a discrete valuation ring. Call $\ord_y$ its valuation at
$y$. Suppose $\ord_y$ is normalized, \ie $\ord_y(t_y)=1$ for each
$t_y\in\m_{\overline{X},y}\setminus\m_{\overline{X},y}^2$. For an ideal $I$ in
$\Or_{\overline{X},y}$ write $\ord_y(I)=\min\{\ord_y(a);\,a\in I\}$. Finally, observe that
since $\overline{X}$ is regular and $\K$ is perfect, $\overline{X}$ is smooth and the
canonical map $c_{\overline{X}}$, defined in Section \ref{subsec:complementary}, is an isomorphism.

\begin{defn}
\begin{enumerate}
\item We call a meromorphic form $\alpha\in\Omega^1_{\Ka_X}$ \textit{finite} if it forms a
$\Ka_X$ basis of $\Omega^1_{\Ka_X}$. In other words, $\alpha$ is not identically zero on
each irreducible component of $X$.
\item Fix a finite form $\alpha$. Following \cite{MVS} we introduce the 
\textit{ramification modules}
\begin{align*}
    R^+_x(\alpha) &:= \omega_{X,x}/\omega_{X,x}\cap\alpha\Or_{X,x} \\
    R^-_x(\alpha) &:= \alpha\Or_{X,x}/\omega_{X,x}\cap\alpha\Or_{X,x}.
\end{align*}
\item The above modules are of finite length and we call
\[\rho(\alpha):=\dim_\K R^+_x(\alpha)-\dim_\K R^-_x(\alpha),\]
the \textit{ramification index} of $\alpha$ at $x$.
\end{enumerate}
\end{defn}

\begin{rke}
A differential form $\alpha\in\Omega^1_{X/\K}$ defines a $\Ka_X$-basis of $\Omega^1_{\Ka_X}$
if and only if $\alpha$ is torsion-free in $\Omega^1_{X/\K}$. In this case $\alpha\Or_X$ can
be identified with a submodule of $\omega_X$ and $R^-_x(\alpha)=0$. Observe also that $\nu$ is
birational and thus $\nu$ gives an isomorphism $\Ka_{X}\to\nu_*\Ka_{\overline{X}}$ and
$d\nu$ identifies $\Omega^1_{\Ka_X}$ with $\nu_*\Omega^1_{\Ka_{\overline{X}}}$. Therefore, any
finite form $\alpha$ on $X$ can be considered as a finite form on $\overline{X}$.
\end{rke}

Following \cite{Pi} we define the \textit{ramification ideal} $R_X:=\Fitt_0
(\Omega^1_{\overline{X}/X})$, which is the $0$-th Fitting ideal of the module of relative
Kähler differentials of $\overline{X}$ over $X$. It is an ideal in $\Or_{\overline{X}}$.
By analogy with \cite[Theorem 3-7-23, p.\ 114]{Weiss}, we introduce the \textit{differential
multiplicity} as the Hilbert-Samuel multiplicity
$e(R_X)=e((\nu_*R_X)_x)$. Equivalently, it can be defined as
\[e(R_X)=\sum_{y\in\nu^{-1}(x)}\ord_y(R_X\Or_{\overline{X},y}).\]
It can be computed via the \textit{cotangent sequence}
\[\begin{tikzcd}
\nu^*\Omega^1_{X/\K}\ar[r]&\Omega^1_{\overline{X}/\K}\ar[r]&\Omega^1_{\overline{X}/X}\ar[r]&0.
\end{tikzcd}\]
Let $\m_{X,x}$ be the maximal ideal of $\Or_{X,x}$ and let $u_1,\ldots,u_n$ be a system of 
generators for $\m_{X,x}$. Locally around $y$ the image of $\nu^*\Omega^1_{X/\K}\to
\Omega^1_{\overline{X}/\K}$ is generated by the images of $du_1,\ldots,du_n$. Since
$\overline{X}$ is regular and $\K$ is perfect, $\overline{X}$ is smooth and
$\Omega^1_{\overline{X},y}$ are free rank one $\Or_{\overline{X},y}$ modules. Each $du_j$ can
be written in $\Omega^1_{\overline{X},y}$ as $v_j\varpi$ where $v_j\in\Or_{\overline{X},y}$
and $\varpi$ is a free generator of $\Omega^1_{\overline{X},y}$. The quantity $\ord_yv_j$ is
thus well defined and
\begin{equation}\label{eq:order_differentials}
\ord_y(R_X\Or_{\overline{X},y})=\min_{1\leqslant j\leqslant n}\ord_yv_j.
\end{equation}
{\bf Convention.} In this work we often consider $\K$-linear combinations $\sum_{i=1}^Na_ir_i$ of elements
$r_i$ from a ring $R$ containing $\K$. By a {\it general} $\sum_{i=1}^Na_ir_i$ or by a \textit{general linear combination} we
mean that the $a_i$ belong to some non-empty Zariski open subset of
$\A^N_\K$. We also call general any object that depends on such general linear combinations.

\begin{prop}\label{prop:rho_to_ramification}
Suppose $g=\sum_{i=1}^na_iu_i, a_i\in\K$ is general. Then
\[\rho(dg)=2\delta_x+e(R_X).\]
\end{prop}

\begin{proof}
For any finite form $\alpha$ we can compute $\rho(\alpha)$ by pulling it back to
$\overline{X}$. By \cite[Lemma 1.6]{MVS} we get
\[\rho(\alpha)=2\delta_x+\rho(d\nu(\alpha)),\]
where
\[\rho(d\nu(\alpha))=\dim_\K(\nu_*\Omega^1_{\overline{X}})_x/
\alpha(\nu_*\Or_{\overline{X}})_x.\]
However, we have
\[(\nu_*\Omega^1_{\overline{X}})_x=\prod_{y\in\nu^{-1}(x)}\Omega^1_{\overline{X},y}
\quad\text{and}\quad
(\nu_*\Or_{\overline{X}})_x=\prod_{y\in\nu^{-1}(x)}\Or_{\overline{X},y}.\]
Thus
\[\rho(d\nu(\alpha))=
\sum_{y\in\nu^{-1}(x)}\dim_\K\Omega^1_{\overline{X},y}/\alpha\Or_{\overline{X},y}.\]

Consider $\alpha=dg=\sum_ia_idu_i$. By \eqref{eq:order_differentials}, for a general choice
of coefficients $a_i$ we get that for each $y\in\nu^{-1}(x)$
\[\dim_\K\Omega^1_{\overline{X},y}/\alpha\Or_{\overline{X},y}=
\ord_y(R_X\Or_{\overline{X},y}).\]
Hence, for general $g$
\[\rho(d\nu(\alpha))=e(R_X),\]
which concludes the proof.
\end{proof}

\subsection{Tame ramification}
Next we define the Milnor number of $X$ at $x$.

\begin{defn}
Set
\[\mu_x:=2\delta_x-r_x+1\]
and call $\mu_x$ the \textit{Milnor number of} $X$ at $x$.
\end{defn}

Denote by $d:\Or_{X,x}\to\omega_{X,x}$ the composition of the universal differential
$\Or_{X,x}\to(\Omega^1_{X/\K})_x$ and the canonical map $(\Omega^1_{X/\K})_x\to\omega_{X,x}$.
When $\K$ is an algebraically closed field of characteristic zero, one can adapt the proof of
\cite[Proposition. 1.2.1.]{BG} to show that $\mu_x$ is equal to $\dim_\K(\omega_{X,x}/d\Or_{X,x})$.

\begin{defn}
Denote by $m_x$ the multiplicity of $X$ at $x$. It is the Hilbert--Samuel multiplicity of
the maximal ideal $\m_{X,x}$ in $\Or_{X,x}$. We say that the point $x$ is \textit{tame} if
$\mathrm{char}(\K)=0$ or if $\mathrm{char}(\K)=p>0$ and
$p\nmid\ord_y(\m_{X,x}\Or_{\overline{X},y})$ for each $y\in\nu^{-1}(x)$.
\end{defn}

\begin{prop}\label{prop:tame_ramification}
We have $e(R_X)\geqslant m_x-r_x$. Equality holds if and only if $x$ is a tame point.
Furthermore, if $x$ is a tame point, for general $g$ we have
\[\rho(dg)=\mu_x+m_x-1.\]
\end{prop}

\begin{proof}
Consider the induced morphism $\widehat{\Or}_{X,x}\to\widehat{\Or}_{\overline{X},y}$ between
complete local rings. Set
\[\widehat{\Omega}^1_{X,x}:=\Omega^1_{X/\K}\otimes_{\Or_{X}}\widehat{\Or}_{X,x}
\quad\text{and}\quad\widehat{\Omega}^1_{\overline{X},y}:=
\Omega^1_{\overline{X}/\K}\otimes_{\Or_{\overline{X}}}\widehat{\Or}_{\overline{X},y}.\]
Since the modules of Kähler differentials are finitely generated, by
\cite[\href{https://stacks.math.columbia.edu/tag/00MA}{Tag 00MA}]{St} we have
\[\widehat{\Omega}^1_{X,x}=\varprojlim_{n\in\N}\Omega^1_{X,x}/\m_{X,x}^n\Omega^1_{X,x}
\quad\text{and}\quad\widehat{\Omega}^1_{\overline{X},y}=
\varprojlim_{n\in\N}\Omega^1_{\overline{X},y}/\m_{\overline{X},y}^n\Omega^1_{\overline{X},y}.\]
Set $t_y$ a uniformizer for $\Or_{\overline{X},y}$. By the Cohen structure theorem, every
$\gamma\in\widehat{\Or}_{\overline{X},y}$ can be expanded as a formal power series in $t_y$
\[\gamma=\sum_{k\in\N}\gamma_kt_y^k,\:\gamma_k\in\K.\]
We continue to write $\ord_y$ for the order of vanishing along $t_y$. Write $\gamma'$ for
the formal derivative of $\gamma$:
\[\gamma':=\sum_{k\in\N}k\gamma_kt_y^{k-1}.\]
The induced differential morphism $\widehat{\Omega}^1_{X,x}\to
\widehat{\Omega}^1_{\overline{X},y}$
takes any $du$ with $u\in\widehat{\Or}_{X,x}$ and sends it to $\gamma'dt_y$, where $\gamma$
is the expansion of $u$ inside $\widehat{\Or}_{\overline{X},y}$. Let $u_1,\ldots,u_n$ be
generators of $\m_{X,x}$ and write $\gamma_1,\ldots,\gamma_n$ for their expansion in
$\widehat{\Or}_{\overline{X},y}$. Then $\gamma'_1,\ldots,\gamma'_n$ generate
$R_X\widehat{\Or}_{\overline{X},y}$ and so
\[\ord_yR_X\Or_{\overline{X},y}=\ord_yR_X\widehat{\Or}_{\overline{X},y}=
\min_{1\leqslant i\leqslant n}\ord_y\gamma'_i.\]
We have the following for $y\in\nu^{-1}(x)$
\[\min_{1\leqslant i\leqslant n}\ord_yu_i=\min_{1\leqslant i\leqslant n}\ord_y\gamma_i\leqslant
1+\min_{1\leqslant i\leqslant n}\ord_y\gamma'_i.\]
By additivity of the multiplicity we have
\begin{align*}
m_x &=\sum_{y\in\nu^{-1}(x)}\min_{1\leqslant i\leqslant n}\ord_yu_i \\
&\leqslant\sum_{y\in\nu^{-1}(x)}1+\min_{1\leqslant i\leqslant n}\ord_y\gamma'_i \\
&=r_x+\sum_{y\in\nu^{-1}(x)}\ord_y(R_X\Or_{\overline{X},y})\\
&=r_x+e(R_X).
\end{align*}
Thus we have $e(R_X)\geqslant m_x-r_x$. To obtain equality, one needs
\[\min_{1\leqslant i\leqslant n}\ord_y\gamma_i=
1+\min_{1\leqslant i\leqslant n}\ord_y\gamma'_i\]
for all $y\in\nu^{-1}(x)$. This is equivalent to $x$ being tame.

Assuming $x$ is a tame point, we combine \cref{prop:rho_to_ramification} with $e(R_X)=m_x-r_x$, to obtain $\rho(dg)=2\delta_x+e(R_X)=(2\delta_x-r_x+1)+m_x-1=\mu_x+m_x-1$.
\end{proof}

\begin{ex}
Consider a prime $p$ and $X$ the plane curve parametrized by $t\mapsto(t^p,t^{p+1})$.
Set the origin $x=0$ as our fixed point. If $\K$ is of characteristic different from $p$,
then $e(R_X)=p-1$, but for $\K=\overline{\mathbb{F}}_p$ we have $e(R_X)=p$.
In both cases $m_x=p$.
\end{ex}

\subsection{Multiplicity formulas}
Because we study local invariants of $X$ at $x$ we can replace $X$ by an affine neighborhood
$U$ of $x$. But $X$ is of finite type over $\K$, so we can embed $U$ into some affine $n$-space
$A=\A^n_\K$. The open set $\nu^{-1}(U)\subset\overline{X}$ is also affine and the induced
morphism $\nu^{-1}(U)\to U$ is the normalization morphism.
We can thus assume that $X$ and $\overline{X}$ are affine. Consider a complete intersection curve
$Z\subset A$ that contains $X$ and that is generically reduced along $X$ (recall that $X$ is  reduced by assumption).
Such $Z$ always exists
(see Proposition \ref{prop:reduced_construction}). Write $W=\overline{Z\setminus X}$. Its ideal in $\Or_Z$ is
given by $\Id_W=((0):_{\Or_Z}\Id_X)$. We can further adopt the set-up of
Section \ref{sec:affine_structure} and suppose that $Z$ is given by equations $F_1,\ldots,F_{n-1}$
and that the rightmost minor of the Jacobian
\[\Delta=\det\left(\frac{\partial F_i}{\partial x_j}
\right)_{\substack{1\leqslant i\leqslant n-1\\ 2\leqslant j\leqslant n}}\]
does not vanish on each irreducible component of $X$.

For the reader's convenience, we briefly summarize the Jacobian-type objects appearing in the remainder of the paper. The intrinsic Jacobian ideal
$\Jac(X,x)=\Fitt_{1}(\Omega^1_{X,x})$
is an ideal of $\mathcal O_{X,x}$ while $\Jac(Z,x)=\Fitt_{1}(\Omega^1_{Z,x})\mathcal{O}_{X,x}$ is the restriction of the Jacobian ideal of a complete intersection $Z$ containing $X$, and therefore depends on the choice of $Z$. However, as we will see below, for a general choice of $Z$, these ideals have the same integral closure in $\mathcal{O}_{X,x}$, and hence the same Hilbert--Samuel multiplicity. 
Both Jacobian ideals can be computed from the equations of $X$ and $Z$ in $A$ using the maximal minors of their respective Jacobian matrices. The Hilbert--Samuel multiplicities of these Jacobian ideals admit a geometric interpretation as the degrees of the exceptional divisors of the blowups with centers at these ideals. This viewpoint will play an important role in Section \ref{Nash}.

The normalization $\overline{X}$ gives rise to two additional invariants. The first is the multiplicity $e(R_X)$ of the ramification ideal $R_X=\Fitt_0
(\Omega^1_{\overline{X}/X})$ in $\mathcal{O}_{\overline{X}}$; the second is the delta invariant $\delta_x$. Both are related to the ramification index of the form $dg$, where $g$ is a general linear form, as established in Proposition \ref{prop:tame_ramification}.  

In this subsection we connect the ramification invariants defined in the previous section
with two multiplicities: the Hilbert--Samuel multiplicity  $e(\Jac(Z,x))$  and the intersection multiplicity $I_x(X,W)$.

\begin{lm}\label{lm:jacobian_minor}
With the above assumptions, we have
\[e(\Jac(Z,x))=\dim_\K\Or_{X,x}/\Delta\Or_{X,x}.\]
\end{lm}

\begin{proof}
Suppose $\mathcal{M} \subset \mathcal{N} \subset \mathcal{O}_{X,x}^p$ are finitely generated $\mathcal{O}_{X,x}$-modules. Denote by $\mathcal{R}(\mathcal{M})$ and $\mathcal{R}(\mathcal{N})$ the Rees algebras of $\mathcal{M}$ and $\mathcal{N}$, respectively. We say that $\mathcal{M}$ is a {\it reduction} of $\mathcal{N}$ if $\mathcal{R}(\mathcal{N})$ is an integral extension of $\mathcal{R}(\mathcal{M})$ (see \cite[Sections 16.1 and 16.2]{SH}). Suppose that $\mathcal{N}$ has constant rank $e$, i.e.\ $\mathrm{rk}(\mathcal{N} \otimes \kappa(\eta))=e$, where $\eta$ runs over the generic points of the irreducible components of $X$.  Set $(X,x):=\mathrm{Spec}(\mathcal{O}_{X,x})$. Consider the structure morphism $c_{\mathcal{N}}:\mathrm{Proj}(\mathcal{R}(\mathcal{N})) \rightarrow (X,x)$. The generic fiber of $c_{\mathcal{N}}$ has dimension $e-1$. By the dimension formula $\dim \mathrm{Proj}(\mathcal{R}(\mathcal{N}))=\dim X+e-1=e$. Hence $\dim c_{\mathcal{N}}^{-1}(x)=e-1$. Fix a generating set of $\mathcal N$ and denote its cardinality by
$g_{\mathcal N}$. This gives us an embedding  $\mathrm{Proj}(\mathcal{R}(\mathcal{N})) \subset (X,x) \times \mathbb{P}^{g_{\mathcal{N}}-1}$. 
Now suppose that $\mathcal{M}$ is generated by $e$ general
$\K$-linear combinations of the generators of $\mathcal{N}$. Consider the ideal of $\mathcal{R}(\mathcal{N})$ generated by the image of
$\mathcal{M}$ in degree $1$; we denote it by
$\mathcal{M}\mathcal{R}(\mathcal{N})$. Since $\dim c_{\mathcal N}^{-1}(x)=e-1$,
a general choice of the $e$ generators yields $e$ hyperplane
sections of $c_{\mathcal N}^{-1}(x)$ whose intersection is empty.
Equivalently,
$\mathbb V(\mathcal M\mathcal R(\mathcal N))
\cap c_{\mathcal N}^{-1}(x)=\varnothing.$
It follows that the radical of
$\mathcal M\mathcal R(\mathcal N)$ is the irrelevant ideal of
$\mathcal R(\mathcal N)$. Hence $\mathcal R(\mathcal N)$ is an
integral extension of $\mathcal R(\mathcal M)$, i.e.\
$\mathcal M$ is a reduction of $\mathcal N$ by \cite[Theorem 16.2.3]{SH}.
 
Denote by $J(Z)$ the $\Or_{X,x}$-submodule of $\Or_{X,x}^{\,n-1}$
generated by the columns of the Jacobian matrix of $Z$ restricted to $X$. Since $J(Z)$ has rank $n-1$, the discussion above shows that, after a
general linear change of coordinates, we may assume that the submodule
$J(Z)'$ generated by the columns corresponding to
$x_2,\ldots,x_n$ is a reduction of $J(Z)$.

By \cite[Theorem 16.3.1]{SH} and \cite[Corollary 1.8]{Gaff} the ideal
$\Fitt_0(\Or_{X,x}^{n-1}/J(Z)')$ is a reduction of $\Jac(Z,x)=\Fitt_0(\Or_{X,x}^{n-1}/J(Z))$.
By the Rees criterion for the Hilbert-Samuel multiplicity (see \cite[Proposition 11.2.1
and Theorem 11.3.1]{SH}) we have
\[e(\Jac(Z,x))=e(\Fitt_0(\Or_{X,x}^{n-1}/J(Z)')).\]
The ideal $\Fitt_0(\Or_{X,x}^{n-1}/J(Z)')=\Delta\Or_{X,x}$ is principal and $(X,x)$ is reduced. Thus
\[e(\Fitt_0(\Or_{X,x}^{n-1}/J(Z)'))=
\dim_\K\Or_{X,x}/\Delta\Or_{X,x}\]
concluding the proof.
\end{proof}


\begin{thm}\label{thm:rho_to_multiplicity}
Suppose $Z$ is generically reduced along $X$.  Then 
\begin{equation}\label{key local}
e(\Jac(Z,x))-I_x(X,W)=2\delta_x+e(R_X).
\end{equation}
If $Z$ is general, then $e(\Jac(X,x))=e(\Jac(Z,x))$. The intersection
multiplicity $I_x(X,W)$ for general $Z$ is then an intrinsic invariant
that we denote by $\cid(X,x)$. Thus \eqref{key local} gives
\begin{equation}\label{eq:mult_to_ramif}
e(\Jac(X,x))-\cid(X,x)=2\delta_x+e(R_X).
\end{equation}
Furthermore, if $x$ is a tame point, then
\begin{equation}\label{eq:mult_to_Milnor}
e(\Jac(X,x))-\cid(X,x)=\mu_x+m_x-1.
\end{equation}
\end{thm}
\begin{proof}
First, we will show that 
\begin{equation}\label{ram-mult}
\rho(dg)=e(\Jac(Z,x))-I_x(X,W),
\end{equation}
where $g\in\Or_{X,x}$ is a general $\K$-linear combination of $x_1,\ldots,x_n\in\m_{X,x}$. Up to a linear change of coordinates, we can assume that $g=x_1$.
By Corollary \ref{eq:complete_intersection}, we have the following inclusions
\[\Delta\Or_{X,x}dx_1\subset\Delta\omega_{X,x}=\Id_W\Or_{X,x}dx_1\subset\Or_{X,x}dx_1.\]
By the additivity of length we obtain
\begin{equation}\label{eq:additivity}
\dim_\K\Or_{X,x}dx_1/\Delta\Or_{X,x}dx_1=\dim_\K\Or_{X,x}dx_1/\Id_W\Or_{X,x}dx_1+
\dim_\K\Delta\omega_{X,x}/\Delta\Or_{X,x}dx_1.
\end{equation}
Since $dx_1$ is a $\Ka_{X}$ basis of $\Omega^1_{\Ka_X}$ we have
\begin{equation}\label{eq:cid}
\dim_\K\Or_{X,x}dx_1/\Id_W\Or_{X,x}dx_1=\dim_\K\Or_{X,x}/\Id_W\Or_{X,x}=I_x(X,W).
\end{equation}
Since $\Delta\in\Ka_X^\times$ we have
\begin{equation}\label{eq:rho}
\dim_\K\Delta\omega_{X,x}/\Delta\Or_{X,x}dx_1=\dim_\K\omega_{X,x}/\Or_{X,x}dx_1=\rho(dx_1).
\end{equation}
Combining \eqref{eq:additivity}, \eqref{eq:cid} and \eqref{eq:rho} we get
\begin{equation}\label{eq:total}
\dim_\K\Or_{X,x}/\Delta\Or_{X,x}=I_x(X,W)+\rho(dg).
\end{equation}
By Lemma \ref{lm:jacobian_minor} we obtain
\begin{equation}\label{eq:mult}
\dim_\K\Or_{X,x}dx_1/\Delta\Or_{X,x}dx_1=\dim_\K\Or_{X,x}/\Delta\Or_{X,x}=e(\Jac(Z,x)).
\end{equation}
Combining \eqref{eq:total} and \eqref{eq:mult} we obtain (\ref{ram-mult}). For $g$ equal to a general $\K$-linear combination of generators of $\m_{X,x}$ we can
apply Proposition \ref{prop:rho_to_ramification} and (\ref{ram-mult}) to get (\ref{key local}).

By \cite[Proposition 2.3]{BGR} applied to $(V,0)=(X,x)$, $M=J(X)$, $M_e=J(Z)$ and $S_1=\{x\}$, for general $Z$ the Hilbert-Samuel multiplicities of
$\Jac(Z,x)$ and $\Jac(X,x)$ are equal. Once $Z$ is chosen as above, we apply
(\ref{key local}) to obtain
\eqref{eq:mult_to_ramif}. Since $e(\Jac(X,x)),\delta_x$ and $e(R_X)$ are intrinsic
invariants, so is $\cid(X,x)$. Identity (\ref{eq:mult_to_ramif}) follows directly from
(\ref{key local}). Suppose $x$ is a tame point. Then Proposition \ref{prop:tame_ramification},
(\ref{ram-mult}) and (\ref{eq:mult_to_ramif}) yield \eqref{eq:mult_to_Milnor}. 
\end{proof}

There are two cases in which $I_x(X,W)$ can be computed directly from the Jacobian ideals of $X$ and $Z$.

\begin{cor}\label{smooth curve}
Let $x\in X$ be a smooth point. Then
\[I_x(X,W)=\dim_{\K}\Or_{X,x}/\Jac(Z,x).\]
\end{cor}

\begin{proof}
Because $X$ is smooth, $\delta_x=e(R_X)=0$. By (\ref{key local}) we have
\[e(\Jac(Z,x))=I_x(X,W).\]
Since $\Or_{X,x}$ is a discrete valuation ring we have
\[e(\Jac(Z,x))=\dim_{\K}\Or_{X,x}/\Jac(Z,x)\]
giving us the desired result.
\end{proof}

\begin{cor}\label{lci}
Assume $X$ is a local complete intersection at $x$. Then
\[I_x(X,W)=e(\Jac(Z,x))-e(\Jac(X,x)).\]
\end{cor}

\begin{proof}
Since both $(X,x)$ and $(Z,x)$ are complete intersections in $(\A^n_\K,x)$, we may 
apply (\ref{key local}) to each of them to obtain
\begin{align*}
    e(\Jac(Z,x))-I_x(X,W) &= 2\delta_x+e(R_X), \\
    e(\Jac(X,x)) &=2\delta_x+e(R_X)
\end{align*}
where, in the second case, we use the fact that the corresponding residual curve is empty. Subtracting the second identity from the first yields the desired result.
\end{proof}
Corollary \ref{lci} provides a particularly effective way of computing $I_x(X,W)$ using only Jacobian ideals when $X$ is a reduced local complete intersection. Proposition \ref{jacobian conservation} \rm{(ii)} is the global counterpart of this result, showing that the same Jacobian-theoretic approach yields a computationally effective formula for the arithmetic genus of a reduced projective local complete intersection curve.

The results of this section provide a local interpretation of the complete intersection discrepancy in terms of classical invariants of a curve singularity. In the next section, we show that this invariant also admits a geometric interpretation through the Nash blowup.


\section{The center of the Nash blowup of a Gorenstein scheme}\label{Nash}

In this section $X$ is a reduced scheme of finite type over $\K$ and of pure dimension $d$. Recall that the {\it Nash blowup} $X' \rightarrow X$ is the closure of the rational section $X \dashrightarrow  \mathrm{Grass}_{d}(\Omega_{X/\K}^1)$ defined over the smooth locus of $X$.
In \cite[Proposition 1, p.\ 508]{Pi} Piene defines the $\omega$-jacobian ideal $J$ as
\[J:=\mathrm{Ann}(\mathrm{Coker}(\Omega^d_{X/\K}\xrightarrow{c_X}\omega_X))\]
and shows that if $X$ is a local complete intersection, then $J=\Jac(X)$ is the usual Jacobian ideal. Moreover, she shows
that when $X$ is Gorenstein, $J$ is the center of the Nash blowup $X'\to X$  \cite[Theorem 2, p.\ 516]{Pi}. We give an explicit expression of this
ideal locally, thus answering a question formulated in \cite[Remark 1), p.\ 516]{Pi}. For $X$ normal and Gorenstein, this description of $J$ has already been provided in \cite[Corollary 9.3]{EM}.

Fix an affine cover $(U_i)_{i\in I}$ of $X$. If furthermore $X$ is Gorenstein, we can
suppose that the $U_i$ trivialize $\omega_X$, i.e.\  for every $i$ in $I,\,\omega_X|_{U_i}
\simeq\Or_{U_i}$. Identify $X$ with an affine open from this cover. Set $X:=\mathrm{Spec}(R)$.
We can embed $X$ in some affine $n$-space $\A^n_{\K}$ and construct a complete intersection
$Z\subset \A^n_{\K}$ that contains $X$ and is generically reduced along $X$ (see Proposition \ref{prop:reduced_construction}). Identify  $\mathrm{Jac}(Z)$ and  $\mathrm{Jac}(X)$ with their respective images in $R$.



Suppose $X$ is cut out by the equations $f_1,\ldots,f_r\in\K[x_1,\ldots,x_n]$. Denote by $L_1,\ldots,L_p$
the $e$-element subsets of $\{1,\ldots,r\}$. For each $i=1,\ldots,p$ pick a  general
$A_i\in\mathrm{Mat}(e\times r,\K)$ so that the complete intersection $Z_i$ cut out by the
equations of $A_i(f_1,\ldots,f_r)^T$ is generically reduced along $X$ for each $i$.

\begin{lm}\label{lm:vary_CIA}
For general $Z_i$ we have $\sum_{i=1}^p\Jac(Z_i)=\Jac(X)$.
\end{lm}

\begin{proof}
Let $K\subset\{1,\ldots,n\}$ be an $e$-element set. 
Write  $[J(Z_i)]_{K}$ for the $e\times e$ minor of the Jacobian matrix of $Z_i$
with columns in $K$, $[A_i]_{L_j}$ for the $e\times e$ minor of $A$ with columns in $L_j$, and
$[J(X)]_{L_j,K}$ for the $e\times e$ minor of the Jacobian matrix of $X$ with rows in $L_j$
and columns in $K$. For each $i=1, \ldots, p$, the generalized Cauchy-Binet formula gives us

\begin{equation}\label{eq:CauchyBinet}
[J(Z_i)]_{K}=\sum_{j=1}^p[A_i]_{L_j}[J(X)]_{L_j,K}.
\end{equation}
Set $\mathbf{J}_{Z,K}:=([J(Z_1)]_{K},\ldots,[J(Z_p)]_{K})^T$ and $\mathbf{J}_{X,K}:=([J(X)]_{L_1,K},\ldots,[J(X)]_{L_p,K})^T$. Define the square matrix $\mathbf{A}:=([A_i]_{L_j})_{1\leqslant i,j\leqslant p}$.
The relations in \eqref{eq:CauchyBinet} can be written in a matrix form as
\[\mathbf{J}_{Z,K}=\mathbf{A}\mathbf{J}_{X,K}.\]
We want to impose on the $A_i$ the extra condition
\begin{equation}\label{eq:general_matrix}
\det(\mathbf{A})\neq 0.
\end{equation}
If (\ref{eq:general_matrix}) is satisfied, then
$\mathbf{J}_{X,K}=\mathbf{A}^{-1}\mathbf{J}_{Z,K}$ and thus each minor $[J(X)]_{L_j,K}$
is a linear combination of the minors $[J(Z_i)]_{K}$. Since $K$ is arbitrary,
$\sum_{i=1}^p \Jac(Z_i)$ would then contain all the $e\times e$ minors of $J(X)$ and so
$\sum_{i=1}^p\Jac(Z_i)=\Jac(X)$. It remains to show that the condition
\eqref{eq:general_matrix} is general. It is enough to show it is a non-zero polynomial in the
entries of the matrices $A_i$. To test this we construct matrices $A_i^0 \in\mathrm{Mat}
(e\times r,\K)$ such that the corresponding $\mathbf{A}^0$ built from their $e\times e$ minors
verifies $\det(\mathbf{A}^0)\neq0$. Suppose $L_i=\{k_1^i,\ldots,k_e^i\}$ with
$k_1^i<\ldots<k_e^i$. For each $i=1,\ldots,p$ define $A^0_i$ by
\[(A^0_i)_{\ell,k}:=\left\{\begin{array}{ll}
    \delta_{\ell,j} & \text{if }k=k_j^i  \\
    0 & \text{otherwise.}
\end{array}\right.\]
In other words, for $j=1,\ldots,e$ the $k_j^i$-th column of $A^0_i$ has a $1$ in the $j$-th
row and is zero everywhere else. A direct computation shows that $[A^0_i]_{L_j}=\delta_{i,j}$,
so that $\mathbf{A}^0=I_p$ is the identity matrix and $\det(\mathbf{A}^0)=1$.
This concludes our proof.
\end{proof}

As usual, set $W:=\overline{Z\setminus X}$. Denote by $I_W$ the ideal of $W$ in $R$. From Remark
\ref{rke:affine_structure} we have $\Jac(Z)\subset I_W$.

\begin{prop}\label{cor:Piene}
With the above notations we have
\[J=(\Jac(Z):_{R} I_W).\]
Furthermore, if $X$ is Gorenstein, then $J=\Jac(X)$ if and only if $X$ is a local
complete intersection.
\end{prop}

\begin{proof}
By Remark \ref{rke:affine_structure} and Corollary \ref{eq:complete_intersection}, since
$\Delta\in\Ka_X^\times$ and $dx_1\wedge\cdots\wedge dx_d$ is a basis for $\Omega^d_{\Ka_X}$,
we have
\[J=\mathrm{Ann}(\mathrm{Coker}(\Jac(Z)\hookrightarrow I_W))=
(\Jac(Z):_{R} I_W).\]

Suppose $X$ is Gorenstein. The question is local, so we can work locally at a singular point
$x\in X$.  We have that $\omega_X$ is invertible, and so $I_W$ is principal by Corollary
\ref{eq:complete_intersection}. Call $h$ a generator of $I_W$. Since $X$ and $W$
have no common irreducible component, we have $h\in\Ka_X^\times$. Suppose by contradiction
that $X$ is not a complete intersection at $x$, so that necessarily $h\in\m_{X,x}$.
We have $J(h)=\Jac(Z,x)$. By Lemma \ref{lm:vary_CIA} we construct $p$ complete intersections $Z_i$
such that $\sum_i\Jac(Z_i)=\Jac(X)$. For each $i$ set $h_i\in\m_{X,x}$ such that
$J(h_i)=\Jac(Z_i,x)$ and define $\q=\sum_{i=1}^ph_i\Or_{X,x}$. We clearly have
$\q\subset\m_{X,x}$ and $J\q=\Jac(X)$. If $J\subset\Jac(X)$, then $J\q=J$.
By Nakayama's lemma $J=(0)$, which is impossible.

The converse is already established by Piene. In fact, she proved that if $X$ is a local
complete intersection, then $J=\Jac(X)$ \cite[Proposition 1, p.\ 508]{Pi}.
\end{proof}

Assume $(X,x)$ is a curve. Consider the blowup $\mathrm{Bl}_{J}(X)$ of $X$ with center $J$.
Denote by $D$ the exceptional divisor. We have that $D \rightarrow x$ is proper and
$D$ is a zero-dimensional cycle $D=\sum_{}m_{p}[p]$ in  $A_0(\mathrm{Bl}_{J}(X))$
(the group of zero cycles modulo rational equivalence). Its degree is defined as
$\mathrm{deg}(D)=\sum_p m_{p}[k(p):\K]=\sum_p m_{p}$, where the last identity follows from
the assumption that $\K$ is algebraically closed (see \cite[Definition 1.4]{Ful}).

\begin{cor}\label{degree} Suppose $(X,x)$ is a reduced  
Gorenstein curve. Then $$\mathrm{deg}(D)=e(\mathrm{Jac}(X,x))-\mathrm{cid}(X,x).$$
\end{cor}
\begin{proof}
Denote by $e(J)$ the Hilbert--Samuel multiplicity of $J$ in $\mathcal{O}_{X,x}$. By \cite{Ram}, and \cite[\S 4.3 and Ex.\ 4.3.4]{Ful} we have $\mathrm{deg}(D)=e(J)$. By Proposition \ref{cor:Piene} we have $\mathrm{Jac}(Z, x)=J(h)$, where $h$ is generator for $I_W$. Because $(X,x)$ is a curve, by \cite[§7, no.\ 1 and 3]{BouAC} we have the following identity of Hilbert--Samuel multiplicities
$$e(\mathrm{Jac}(Z,x))=e(J(h)\mathcal{O}_{\overline{X}})=e(J\mathcal{O}_{\overline{X}})+e((h)\mathcal{O}_{\overline{X}})=e(J)+e((h)).$$ 
For general $Z$ by \cite[Proposition 2.3]{BGR} we have $\overline{\mathrm{Jac}(Z,x)}=\overline{\mathrm{Jac}(X,x)}$. In particular, $e(\mathrm{Jac}(Z,x))=e(\mathrm{Jac}(X,x))$. Because $(X,x)$ is reduced we have $$e((h))=\dim \mathcal{O}_{X,x}/(h)=I_{x}(X,W)=\mathrm{cid}(X,x).$$
Thus $\mathrm{deg}(D)=e(J)=e(\mathrm{Jac}(X,x))-\mathrm{cid}(X,x)$ which is what we wanted. 
\end{proof}

Section \ref{sec:Milnor} showed that the complete intersection discrepancy is determined by classical local invariants, while the present section relates it to the geometry of the Nash blowup. In the next section, we turn to the global theory and show that this invariant appears naturally as the correction term in the genus--degree formula for projective curves.

\section{The genus formula}\label{sec:degree}

In this section we prove the genus formula (\ref{eq:arithmetic genus}) for any
Cohen-Macaulay projective curve $X\subset\Pj_\K^n$ embedded in a complete intersection
$Z\subset\Pj_\K^n$ defined by homogeneous polynomials $F_1,\ldots,F_{n-1}$ of respective
degrees $d_1\geqslant\ldots\geqslant d_{n-1}$. No assumption on the characteristic of $\K$
is needed. We use Proposition \ref{thm:adjunction} to compute the Euler characteristic of
$\omega_{X/\K}$. The arithmetic genus $p_a(X)$ is by definition equal to $1-\chi(X,\Or_X)$,
where $\chi(X,\Or_X):=\dim_{\K}H^{0}(X,\mathcal{O}_X)-\dim_{\K}H^{1}(X,\mathcal{O}_X)$ is the
Euler characteristic of the structure sheaf $\Or_X$.

\subsection{A degree formula for tensor products}

We  generalize
\cite[\href{https://stacks.math.columbia.edu/tag/0AYV}{Tag 0AYV}]{St} in a straightforward manner to any
equidimensional proper $\K$-scheme of dimension $1$. This includes
reducible such schemes. We follow a similar \textit{dévissage}
argument as the one used in the proof of \cite[\href{https://stacks.math.columbia.edu/tag/0AYV}{Tag 0AYV}]{St}.
Recall that for $\mathcal{Y}$ a proper $\K$-scheme of dimension $\leqslant1$, the
\textit{degree} of a locally free $\Or_{\mathcal{Y}}$-module $\mathcal{E}$, of constant rank
$\mathrm{rk}(\mathcal{E})$ is defined as
\[\deg(\mathcal{E}):=\chi(\mathcal{Y},\mathcal{E})-\mathrm{rk}(\mathcal{E})\chi(\mathcal{Y},\Or_{\mathcal{Y}}).\]

\begin{prop}\label{prop:ranks}
Consider a field $\K$ and $\mathcal{Y}$ a proper $\K$-scheme, equidimensional of
dimension $1$. Denote by $\mathcal{Y}_1,\ldots,\mathcal{Y}_s$ the irreducible components of $\mathcal{Y}$. Let
$\mathcal{E}$ be a locally free $\Or_\mathcal{Y}$-module on $\mathcal{Y}$ of constant rank
$\mathrm{rk}(\mathcal{E})$ and $\mathcal{F}$ a coherent $\Or_{\mathcal{Y}}$-module.
Define the generic rank of $\mathcal{F}$ at $\mathcal{Y}_i,\,i=1,\ldots,s$ as
\[r_{\xi_i} (\mathcal{F}):=\mathrm{length}_{\Or_{\mathcal{Y},\xi_i}}(\mathcal{F}_{\xi_i}),\]
where $\xi_i$ is the generic point of $\mathcal{Y}_i$. We have the following formula
\begin{equation}
\chi(\mathcal{Y},\mathcal{E}\otimes\mathcal{F})=
\sum_{i=1}^sr_{\xi_i}(\mathcal{F})\deg(\mathcal{E}|_{\mathcal{Y}_i^{\mathrm{red}}})+
\mathrm{rk}(\mathcal{E})\chi(\mathcal{Y},\mathcal{F}).
\label{eq:ranks}
\end{equation}
\end{prop}

\begin{proof}
We say that $\mathcal{F}$
satisfies the property $\mathcal{P}$ if $\mathcal{F}$ verifies
\eqref{eq:ranks}. We wish to apply the general \textit{dévissage} theorem,
Lemma 30.12.6 \cite[\href{https://stacks.math.columbia.edu/tag/01YI}{Tag 01YI}]{St}.
Part (1) of Lemma 30.12.6 is verified because Euler characteristics and the
quantities $r_{\xi_i}(-)$ are additive. Since $\mathcal{E}$ is locally free it is
flat and thus both left-hand side and right-hand side of \eqref{eq:ranks} are
additive in $\mathcal{F}$. To show part (2) of Lemma 30.12.6 is verified, we
consider any integral closed subscheme $Y\subset\mathcal{Y}$. Write $j:Y\hookrightarrow\mathcal{Y}$
for the closed immersion morphism. Consider $\mathcal{G}:=j_*\Or_Y$. Suppose
$Y=\{p\}$ is a closed point $p\in\mathcal{Y}$. Property $\mathcal{P}$ holds by part (5) from
Lemma 33.33.3 in
\cite[\href{https://stacks.math.columbia.edu/tag/0AYT}{Tag 0AYT}]{St} and the fact
that all $r_{\xi_i}$ are $0$ in this case.

Set $Y=\mathcal{Y}_i^{\mathrm{red}}$. Points (a), (b) and (c) of part (2) in the dévissage
theorem are clearly verified as $\mathcal{G}_{\xi_j}$ is $\kappa(\xi_i)$ if $i=j$
and $0$ otherwise. Thus $r_{\xi_i}(\mathcal{G})=\delta_{i,j}$. To check point (d) of
part (2) of the dévissage theorem we use
\cite[\href{https://stacks.math.columbia.edu/tag/089W}{Tag 089W}]{St}
\[\chi(\mathcal{Y},\mathcal{G})=\chi(\mathcal{Y},j_*\Or_Y)=\chi(Y,\Or_Y).\]
By the projection formula
\cite[\href{https://stacks.math.columbia.edu/tag/01E8}{Tag 01E8}]{St}
\[\mathcal{E}\otimes\mathcal{G}=\mathcal{E}\otimes j_*\Or_Y=
j_*(j^*\mathcal{E}\otimes\Or_Y)=j_*j^*(\mathcal{E})=j_*(\mathcal{E}|_{Y}),\]
so again by \cite[\href{https://stacks.math.columbia.edu/tag/089W}{Tag 089W}]{St}
\[\chi(\mathcal{Y},\mathcal{E}\otimes\mathcal{G})=\chi(Y,\mathcal{E}|_{Y}).\]
Property $\mathcal{P}$ is now equivalent to the definition of the degree of
$\mathcal{E}|_{Y}$:
\[\deg(\mathcal{E}|_{Y})=\chi(Y,\mathcal{E}|_{Y})-
\mathrm{rk}(\mathcal{E}|_{Y})\chi(Y,\Or_{Y}).\]
Indeed $\mathrm{rk}(\mathcal{E}|_{Y})=\mathrm{rk}(\mathcal{E})$.
\end{proof}

\subsection{The arithmetic genus formula}
Below we prove Theorem \ref{thm:degree} (\ref{eq:arithmetic genus}).
As usual denote $W=\overline{Z\setminus X}$ and $\iota:X\hookrightarrow Z,
\,\iota':W\hookrightarrow Z$ the respective closed immersions.
Write $Z=\bigcup_{i=1}^l Z_i$ for the irreducible decomposition of $Z$. Up to
re-indexing, suppose $Z_1,\ldots,Z_k$ correspond to the irreducible components of $X$
and $Z_{k+1},\ldots,Z_l$ to the irreducible components of $W$. For each $i$ denote by $\xi_i$
the generic point of $Z_i$. We now successively
apply \cite[\href{https://stacks.math.columbia.edu/tag/089W}{Tag 089W}]{St}, our
adjunction-type formula (\ref{eq:adjunction}) and Proposition \ref{prop:ranks} to $\mathcal{Y}=Z$,
$\mathcal{E}=\omega_{Z/\K}$ and $\mathcal{F}=\Id_W$. We have
\begin{align}\label{eq:Euler}
\chi(X,\omega_{X/\K})
&=\chi(Z,\iota_*\omega_{X/\K}) \\
&=\chi(Z,\Id_W\otimes\omega_{Z/\K}) \notag \\
&=\sum_{i=1}^l r_{\xi_i}(\Id_W)\deg(\omega_{Z/\K}|_{Z_i^{\mathrm{red}}})+
\chi(Z,\Id_W). \notag
\end{align}
Recall that $\omega_{Z/\K}$ is locally free of rank $1$ and by \cite[Ch.\ III, Thm. 7.11]{HaAG}
\[\omega_{Z/\K}=\omega_{\Pj_\K^n}|_Z\otimes
\left(\bigwedge^{n-1}\Id_Z/\Id_Z^2\right)^\vee.
\label{eq:adjunction_complete_intersection}\]
By \cite[Ch.\ II, Ex.\ 8.20.1]{HaAG} we have $\omega_{\Pj_\K^n}\simeq
\Or_{\Pj_\K^n}(-n-1)$, so
$\omega_{\Pj_\K^n}|_Z\simeq\Or_Z(-n-1).$
The coherent module $\Id_Z/\Id_Z^2$ is locally free of rank $n-1$: the sheaf morphism
\[\begin{tikzcd}
\Or_{Z}(-d_1)\oplus\ldots\oplus
\Or_{Z}(-d_{n-1})\ar[r] & \Id_Z/\Id_Z^2,\quad
e_i \ar[r,maps to] & \bar f_i,
\end{tikzcd}\]
where $e_i$ is the basis element of the $i$-th factor in the above direct sum is an isomorphism.
Therefore, $\bigwedge^{n-1}\Id_Z/\Id_Z^2$ is isomorphic to $\Or_Z(-d_1-\cdots-d_{n-1})$. We obtain
\begin{equation}\label{eq:omega_Z}
\omega_{Z/\K}\simeq\Or_Z(d_1+\cdots+d_{n-1}-n-1).
\end{equation}
In particular, $$\deg(\omega_{Z/\K}|_{Z_i^{\mathrm{red}}})=
\deg\Or_{Z_i^{\mathrm{red}}}(d_1+\cdots+d_{n-1}-n-1)=
\deg(Z_i^{\mathrm{red}})(d_1+\cdots+d_{n-1}-n-1).$$

We now compute the $r_{\xi_i}$. Consider the canonical short exact sequence
\begin{equation}\label{SES:ideal}
\begin{tikzcd}
0\ar[r]&\Id_W \ar[r] & \Or_Z \ar[r] & \iota'_*\Or_W \ar[r] & 0.
\end{tikzcd}
\end{equation}
There is a corresponding short exact sequence for stalks at the $\xi_i$, thus, by
additivity of $\kappa(\xi_i)$ vector spaces, we have
\[r_{\xi_i}(\Id_W)=r_{\xi_i}(\Or_Z)-r_{\xi_i}(\iota'_*\Or_W).\]
For any equidimensional ring $A$ and ideal $I$ that is the intersection of certain
primary ideals from the minimal primary decomposition of $(0)$ in $A$,
the localizations of $A/I$ at any $q\in\mathrm{Min}(A)$ are
\[(A/I)_q\simeq\begin{cases}
    A_q & \text{if }q\in\mathrm{Min}(I) \\
    0 & \text{otherwise.}
\end{cases}\]
Thus $r_{\xi_i}(\Or_Z)=\mathrm{length}_{\Or_{Z,\xi_i}}(\Or_{Z,\xi_i})$ and
$r_{\xi_i}(\iota'_*\Or_W)=\mathrm{length}_{\Or_{Z,\xi_i}}(\Or_{Z,\xi_i})$ or $0$ whether
$\xi_i$ is a generic point of an irreducible component of $W$ or not. So,
\begin{align}\label{eq:sum}
\sum_{i=1}^l r_{\xi_i}(\Id_W)\deg(\omega_{Z/\K}|_{Z_i})
&=\sum_{i=1}^k \mathrm{length}_{\Or_{Z,\xi_i}}(\Or_{Z,\xi_i})
\deg(Z_i^{\mathrm{red}})(d_1+\cdots+d_{n-1}-n-1) \notag \\
&=\deg(X)(d_1+\cdots+d_{n-1}-n-1), 
\end{align}
as $\sum_{i=1}^k\mathrm{length}_{\Or_{Z,\xi_i}}(\Or_{Z,\xi_i})
\deg(Z_i^{\mathrm{red}})=\deg(X)$ by Bézout (one can alternatively use
\eqref{eq:ranks} with $\mathcal{Y}=X$, $\mathcal{E}=\Or_X(1)$ and $\mathcal{F}=\Or_X$).
We now use \eqref{SES:ideal} a second time to compute the Euler characteristic of
$\Id_W$
\begin{equation}\label{eq:Euler_W}
\chi(Z,\Id_W)=\chi(Z,\Or_Z)-\chi(Z,\iota'_*\Or_W)=\chi(Z,\Or_Z)-\chi(W,\Or_W).
\end{equation}

Consider now another short exact sequence
\begin{equation}\label{SES:sum}
\begin{tikzcd}
0\ar[r]&\Or_Z \ar[r]& \iota_*\Or_X\oplus \iota'_*\Or_W\ar[r] &\Or_Z/(\Id_X+\Id_W) \ar[r] & 0.
\end{tikzcd}
\end{equation}
The first map is simply the sum of the structure maps $\Or_Z\to \iota_*\Or_X$ and
$\Or_Z\to \iota'_*\Or_W$. The second map is the difference map on sections
$(\alpha,\beta)\mapsto\alpha-\beta\mod\Id_X+\Id_W$.
Observe that the coherent sheaf $\Or_Z/(\Id_X+\Id_W)$ has discrete support $X\cap W$. By (2) and (4) in
\cite[\href{https://stacks.math.columbia.edu/tag/0AYT}{Tag 0AYT}]{St} we have

\begin{align}\label{eq:finite_support}
\chi(\Or_Z/(\Id_X+\Id_W))
&=\dim_\K H^0(Z,\Or_Z/(\Id_X+\Id_W)) \\
&=\sum_{z\in X\cap W}\dim_\K\Or_{Z,z}/(\Id_X+\Id_W) \notag \\
&=I(X,W). \notag
\end{align}

By additivity of $\chi$ in \eqref{SES:sum} and
\cite[\href{https://stacks.math.columbia.edu/tag/089W}{Tag 089W}]{St} we have
\begin{align}\label{eq:Euler_gives_cid}
I(X,W)
&=\chi(Z,\iota_*\Or_X\oplus \iota'_*\Or_W)-\chi(Z,\Or_Z) \\
&=\chi(Z,\iota_*\Or_X)+\chi(Z,\iota'_*\Or_W)-\chi(Z,\Or_Z) \notag \\
&=\chi(X,\Or_X)+\chi(W,\Or_W)-\chi(Z,\Or_Z). \notag
\end{align}
In conclusion, by combining \eqref{eq:Euler}, \eqref{eq:sum}, \eqref{eq:Euler_W},
\eqref{eq:finite_support} and \eqref{eq:Euler_gives_cid} we obtain
\begin{equation}\label{eq:Euler-degree}
\chi(X,\omega_{X/\K})-\chi(X,\Or_X)=\deg(X)(d_1+\cdots+d_{n-1}-n-1)-I(X,W).
\end{equation}

Because $X$ is Cohen-Macaulay by \cite[\href{https://stacks.math.columbia.edu/tag/0BS5}{Tag 0BS5}]{St} we have
$\chi(X,\omega_{X/\K})=-\chi(X,\Or_X).$
By definition $\chi(X,\Or_X)=1-p_a(X)$. Therefore, (\ref{eq:Euler-degree}) yields (\ref{eq:arithmetic genus}):

\[p_a(X)=1+\frac{\deg(X)(d_1+\cdots+d_{n-1}-n-1)-I(X,W)}{2}.\]

When $X$ is a Cohen--Macaulay curve that is generically a complete intersection, 
the existence of a complete intersection $Z$ with the desired properties is established in Proposition~\ref{prop:generic_CI_construction}. 
This yields \eqref{eq:arithmetic genus}.
When $X$ is smooth, Proposition \ref{jacobian conservation} \rm{(i)} implies \eqref{eq:cid-jacobian}. We can then compute $g_X$ 
by applying the formula linking the arithmetic and geometric genera of curves (see
\cite[Ch.\ IV, §1,2]{Se} and \cite[Thm. 2, p.\ 190]{Hir57}) to obtain (\ref{eq:genus_X}).
\begin{rke}\label{Plücker} Assume $\mathrm{char}(\K)=0$. We say that $H \subset \mathbb{P}^n$ is a {\it tangent hyperplane} to a smooth point $x \in X$ if $H$ contains the tangent space $T_{X,x}$ of $X$ at $x$. Hyperplanes in $\mathbb{P}^n$ are naturally identified with points in the dual projective space $\check{\mathbb{P}^n}$. Consider the conormal variety $C(X):=\overline{\{(x,H)|x \in X_{\mathrm{sm}}, T_{X,x} \subset H\}} \subset X \times \check{\mathbb{P}^n}$. The projection  $\check{X}$  of $C(X)$ onto the second factor is called the {\it dual variety} of $X$. Because $X$ is a curve, $\check{X}$ is a hypersurface in $\check{\mathbb{P}^n}$. Its degree is called {\it the class} of $X$. The following Plücker formula for the class of $X$ is derived in \cite{RT} 

$$\mathrm{deg}(\check{X})=(d_1+ \cdots + d_{n-1} - n+1)\mathrm{deg}(X)-I(X,W)-\sum_{x \in X_{\mathrm{sing}}}(\mu_x+m_x-1).$$

Combining this formula with (\ref{eq:arithmetic genus}) the second author and Teissier deduced the identity

$$\mathrm{deg}(\check{X})-2\mathrm{deg}(X)=2p_a(X)-2-\sum_{x \in X_{\mathrm{sing}}}(\mu_x+m_x-1),$$
where the two terms on the right-hand side of the formula 
represent a global intrinsic invariant of $X$ and local intrinsic invariants of the singularities of $X$.
\end{rke}


\subsection{Two corollaries}

Note that $I(X,W)=I(W,X)$. A direct consequence of \eqref{eq:arithmetic genus} is
the genus difference formula for directly linked curves \cite[Proposition 3.1]{PS}.

\begin{cor}\label{PS}
Suppose two reduced curves $X$ and $W$ are directly linked via complete intersection
$Z\subset\Pj^n_\K$ cut out by homogeneous equations of degrees
$d_1,\ldots,d_{n-1}$. Then we have the following relation between arithmetic genera and
degrees
\[p_a(X)-p_a(W)=(\deg(X)-\deg(W))\frac{(d_1+\cdots+d_{n-1}-n-1)}{2}.\]
\end{cor}



Finally, we show that the complete intersection discrepancy is constant under flat deformations. Let $X$ be a Cohen--Macaulay curve and $\mathcal{X} \rightarrow T$ be a flat morphism with $T$ connected, $\mathcal{X} \subset \mathbb{P}_{\K}^n \times T$, and $\mathcal{X}_{t_0}=X$ for a closed point $t_0 \in T$. Assume $\mathcal{Z} \rightarrow T$ with $\mathcal{Z} \subset \mathbb{P}_{\K}^n \times T$ is a flat family of complete intersection curves such that $X_t$ is a union of irreducible components of $Z_t$ for each closed point $t \in T$. Set $\mathcal{W}_t:=\overline{\mathcal{Z}_t\setminus \mathcal{X}_t}$. 

\begin{cor}\label{const. cid} We have $t \rightarrow I(\mathcal{X}_t,\mathcal{W}_t)$ is constant for $t \in T$.
\end{cor}
\begin{proof} By \cite[Ch.\ III, Cor.\ 9.10]{HaAG}, $p_{a}(\mathcal{X}_t)$, $\mathrm{deg}(\mathcal{Z}_t)$ and $\mathrm{deg}(\mathcal{X}_t)$ are constant, and so Theorem \ref{thm:degree} (\ref{eq:arithmetic genus}) yields the desired result.
\end{proof}

\section{The complete intersection discrepancy}\label{sec:complete_int}



\subsection{Construction of a complete intersection}\label{construction}
Let $X\subset\Pj^n_\K$ be a Cohen--Macaulay scheme of pure dimension $d$.
Let $I_X=(f_1,\ldots,f_r)\subset S=\K[x_0,\ldots,x_n]$ be the homogeneous ideal of $X$,
and $\mathcal{I}_{X}$ the associated ideal sheaf on $\mathbb{P}^n_\K$. Recall that any
homogeneous prime $\p$ of $S$, corresponds to a point $\eta$ in $\mathbb{P}^n_\K$,
and we have $(\mathcal{I}_X)_\eta=(I_X)_{(\p)}$, where $(-)_{(\p)}$
denotes the degree zero part of the localization at $\p$.
Set $d_i=\deg(f_i)$. Assume $d_1\geqslant d_2\geqslant\ldots\geqslant d_r$.

We henceforth assume that $X$ is generically a complete intersection. By this we mean that, for
each generic point $\eta$ of an irreducible component of $X$, the ideal $(\mathcal{I}_{X})_\eta$
is generated by $n-d$ elements in the local ring $\Or_{\mathbb{P}^n_\mathbbm{k},\eta}$.
Our goal is to construct efficiently, using prime avoidance and linear algebra,
a complete intersection $Z=\mathbb{V}(F_1,\ldots,F_{n-d})$, where $F_i\in (I_X)_{d_i}$ and $Z=X$ at the generic points $\eta$.

Such constructions have been carried out in \cite[Proposition 3.7]{PS}, \cite[Theorem 4.4]{CU}
and \cite[Proposition 2.5]{HU} via different methods. Our goal is to exert explicit control on the type
of ``homogenized'' linear combinations of the $f_1,\ldots,f_r$  that give the equations of $Z$.

If $n=d+1$, then $X$ is a Cohen--Macaulay hypersurface. By
\cite[\href{https://stacks.math.columbia.edu/tag/0BXL}{Tag 0BXL}]{St} $I_X$ is principal.
So we set $Z=X$ and we are done. 

The idea is to construct inductively a regular sequence
$\tilde f_1,\ldots,\tilde f_{n-d}$ by taking suitable
homogeneous combinations of $f_1,\ldots,f_r$ of the same degree,
ensuring at each step that $\tilde f_i$ avoids the minimal primes
of $(\tilde f_1,\ldots,\tilde f_{i-1})$.

Suppose $n \geq d+2$. We are going to select $n-d-1$ linear forms $\ell_1,\ldots,
\ell_{n-d-1}\in S(\Pj_{\K}^n)_1$ and constants $b_{i,j}\in\K,\,1\leqslant i
\leqslant n-d,\,1\leqslant j\leqslant r$ subject to certain general conditions
to be specified below. Select $\ell_1$ such that it avoids the minimal primes of $(f_1)$. Set
\[\Tilde{f}_1=b_{1,1}f_1+b_{1,2}\ell_1^{d_1-d_2}f_2+
\cdots+b_{1,r}\ell_1^{d_1-d_r}f_r.\]
Consider the ideal $I_1:=(f_2, \ell_1^{d_2-d_3}f_3,\ldots,\ell_1^{d_2-d_r}f_r)$.
Suppose there exists a minimal prime  $\p_1$ of $(\Tilde{f}_1)$ that contains $I_1$.
Then either $\p_1$ contains $\ell_1$ and $f_1$ or $\p_1$ contains $(f_1,\ldots,f_r)=I_X$,
which is impossible because $\mathrm{ht}(\p_1)=1$. Then by prime avoidance there exist general
$b_{2,2},\ldots,b_{2,r}\in\K$ such that
\[\Tilde{f}_2=b_{2,2}f_2+b_{2,3}\ell_1^{d_2-d_3}f_3+\cdots+
b_{2,r}\ell_1^{d_2-d_r}f_r\]
avoids the minimal primes of $(\Tilde{f}_1)$. If $n=d+2$ we are done.
If $n>d+2$, select $\ell_2$ such that
$\ell_2$ avoids the minimal primes of $(\Tilde{f}_1,\Tilde{f}_2)$ and continue as follows. Consider the ideal $I_2=(f_3,\ell_2^{d_3-d_4}f_4,\ldots,\ell_2^{d_3-d_r}f_r)$.
Let $\p_2$ be a minimal prime of $(\Tilde{f}_1,\Tilde{f}_2)$.
If $I_2\subset\p_2$, then either $(\ell_2,\Tilde{f}_1,\Tilde{f}_2)\subset\p_2$ or
$I_X\subset\p_2$. Both cases are impossible, because $\mathrm{ht}(\p_2)=2$ and
$\mathrm{ht}(I_X)\geqslant3$. Thus there exist general $b_{3,3},\ldots,b_{3,r}\in\K$ such that
\[\Tilde{f}_3=b_{3,3}f_3+b_{3,4}\ell_2^{d_3-d_4}f_4+\cdots+
b_{3,r}\ell_2^{d_3-d_r}f_r\]
avoids the minimal primes of $(\Tilde{f}_1,\Tilde{f}_2)$. Continuing this process
we obtain a complete intersection $Z=\mathbb{V}(\Tilde{f}_1,\ldots,\Tilde{f}_{n-d})$.
Note that $(\Tilde{f}_1,\ldots, \Tilde{f}_{n-d},f_{n-d+1}, \ldots,f_{r})=(f_{1},\ldots, f_{r})$.

By imposing additional genericity conditions on the coefficients $b_{i,j}$,
we ensure that the resulting complete intersection $Z$ coincides with 
$X$ at the generic point of each irreducible component of  $X$. We then denote by 
$F_1, \ldots, F_{n-d}$ the defining equations of $Z$. When $X$ is affine or $X$ is projective such that $X$ is defined by homogeneous equations of the same
degree, then each $F_i$ can be chosen as a general $\K$-linear combination of $f_1, \ldots, f_r$.

In the special case where $X$ is reduced, we can give a direct construction of $Z$ such that $Z$
is generically reduced along $X$; in particular, this guarantees $Z=X$ generically.
Assume $X=\bigcup_{a=1}^k X_a$ is the decomposition of $X$ into irreducible components.
For each $a=1,\ldots,k$ select a smooth point $z_a\in X_a$ such that $z_a\not \in\mathbb{V}(\ell_j)$
for each $j=1,\ldots,n-d-1$.

\begin{prop}\label{prop:reduced_construction}
Assume $X$ is reduced and consider the Jacobian matrix $J(Z)|_X$ and points $z_a$ as above.
For general $b_{i,j}$ the matrix $J(Z)$ evaluated at $z_a$ is of maximal rank $n-d$.
In particular, $Z$ is generically reduced along $X$.
\end{prop}

\begin{proof}
Note that $J(X)$ is of maximal rank $e=n-d$ at each $z_a$ and by the Leibniz rule
\[
\scalemath{0.90}{
J(Z)|_X=\begin{pmatrix}
b_{1,1} & b_{1,2}\ell_1^{d_1-d_2} & b_{1,3}\ell_1^{d_1-d_3} &
    b_{1,4}\ell_1^{d_1-d_4} & \cdots & \cdots & \cdots &
    b_{1,r}\ell_1^{d_1-d_r} \\
0 & b_{2,2} & b_{2,3}\ell_1^{d_2-d_3} & b_{2,4}\ell_1^{d_2-d_4}
    & \cdots &  \cdots & \cdots & b_{2,r}\ell_1^{d_2-d_r} \\
0 & 0 & b_{3,3} & b_{3,4}\ell_2^{d_3-d_4} & \cdots & \cdots & \cdots &
    b_{3,r}\ell_2^{d_3-d_r} \\
\vdots & \vdots & \vdots & \ddots & & & & \vdots \\
0 &  0 & 0 & \cdots & b_{e,e} & b_{e,e+1}\ell_{e-1}^{d_{e}-d_{e+1}} & \cdots &
    b_{e,r}\ell_{e-1}^{d_{e}-d_r}
\end{pmatrix} J(X).
}\]
Write $b=(b_{i,j})_{1\leqslant i\leqslant n-d,\; i\leqslant j\leqslant r}$. By the generalized
Cauchy--Binet formula (cf.\ (\ref{eq:CauchyBinet})), for each $a=1,\ldots,k$ there exists a non-zero polynomial $g_a$
in the variables $b_{i,j}$ such that $g_a(b)\neq0$ implies that $J(Z)|_X$ has a non-zero
maximal minor at $z_a$. Setting $g=\prod_{a=1}^k g_a$, 
we see that $g(b)\neq0$ implies that $J(Z)|_X$ has rank $n-d$ at every point $z_a$.

Assume moreover that the coefficients $b_{i,j}$ are chosen generally so that $Z$ is a
complete intersection. Then $Z$ is Cohen--Macaulay. Since $J(Z)|_X$ has maximal rank at
each $z_a$, the ideal of maximal minors of $J(Z)$ contains a non-zero divisor in
$\Or_{Z,\eta}$ for the generic point $\eta$ of each irreducible component of $X$.
It follows from \cite[Thm.~18.15(a)]{Ei} that $\Or_{Z,\eta}$ is reduced. Hence $Z$ is
reduced at the generic point of each irreducible component of  $X$, and therefore $Z$ is generically reduced along $X$.
\end{proof}

We return to the general case. When $X$ is reduced,
Proposition~\ref{prop:reduced_construction}
gives a Jacobian-theoretic construction of $Z$
that is generically reduced along $X$, with genericity conditions
that can be checked explicitly using the Jacobian matrix.
For a Cohen--Macaulay generic complete intersection $X$,
the local rings at the generic points need not be reduced,
so we instead impose genericity conditions directly at those
generic points.

\begin{prop}\label{prop:generic_CI_construction}
Suppose $X \subset\Pj_\K^n$ is a Cohen--Macaulay scheme that is generically a complete intersection. 
For the generic point $\eta$ of each irreducible component of $X$, there is a
non-trivial genericity condition on the coefficients $b_{i,j}$ such that
$\mathcal{I}_{Z,\eta}=\mathcal{I}_{X,\eta}$.
\end{prop}

\begin{proof}

Let $\p$ be the homogeneous prime ideal of $S=S(\mathbb{P}^n_\mathbbm{k})$ corresponding to $\eta$.
Thus $\mathcal{I}_{X,\eta}$ is isomorphic to $(I_X)_{(\p)}$, the degree zero component of the
localization $(I_X)_\p$. This ideal is generated by $\frac{f_1}{\varphi_1},\ldots,
\frac{f_r}{\varphi_r}$ for any $\varphi_j\in S(\mathbb{P}^n_\mathbbm{k})_{d_j}\setminus\p$.
Observe that the latter set is never empty, otherwise $\p$ would be the irrelevant ideal. By a similar
remark, it suffices to find a regular sequence $F_1,\ldots,F_{n-d}$, that generates $(I_X)_\p$,
so that $\frac{F_1}{\psi_1},\ldots,\frac{F_{n-d}}{\psi_{n-d}}$ generate $\mathcal{I}_{X,\eta}$,
for any $\psi_i\in S(\mathbb{P}^n_\mathbbm{k})_{d_i}\setminus\p$.

Define $\kappa:=\kappa(\p)=\mathrm{Frac}(S/\p)=S_\p/\p S_\p$, the residual field of $\p$
and $V=I_X\otimes\kappa$. By Nakayama's lemma, $X$ being a complete intersection
at $\eta$ is equivalent to $\dim_{\kappa}V=n-d$. Write $\bar{f}_1,\ldots,\bar{f}_r$
and $\bar{F}_1,\ldots,\bar{F}_{n-d}$ for the respective images of the $f_j$ in $V$ and
$F_i$ in $V$. The $\bar{f}_j$ span a $(n-d)$-dimensional $\kappa$-vector space. Let
$E_1,\ldots,E_{n-d}$ and $e_1,\ldots,e_r$ be the canonical basis of $\kappa^{n-d}$ and
$\kappa^{r}$, respectively. Write $A:\kappa^{n-d}\to\kappa^{r}$ for the $\kappa$-linear map defined by
\[A(E_i)=\sum_{j=i}^r b_{i,j}\bar{\ell}_{k(i)}^{d_i-d_j}e_j,\]
where $k(i)=1$ if $i=1$  and $i-1$ otherwise, and $\bar{\ell}_{k(i)}$ is the image of $\ell_{k(i)}$
inside $\kappa$. We identify $A$ with its matrix in the canonical bases.
Furthermore let $\Phi:\kappa^r\to V$ be the $\kappa$-linear map defined by
$\Phi(e_j)=\bar{f}_j$. Thus $\Phi$ is surjective and $(\Phi A)(E_i)=\bar{F}_i$. The $\bar{F}_i$
span $V$ iff $\Phi A$ is surjective. This is equivalent to $\wedge^{n-d}(\Phi A)$ being a
non-zero map. By functoriality of exterior products $\wedge^{n-d}(\Phi A)=\wedge^{n-d}\Phi\wedge^{n-d}A$. For any subset
$\sigma=\{j_1,\ldots,j_{n-d}\}\subset\{1,\ldots,r\},\:j_1<\ldots<j_{n-d}$, write
$e_\sigma:=e_{j_1}\wedge\cdots\wedge e_{j_{n-d}}$ and similarly $\bar{f}_\sigma:=\bar{f}_{j_1}
\wedge\cdots\wedge\bar{f}_{j_{n-d}}$. Denote by $[A]_\sigma$ the minor of the matrix $A$ determined
by taking the rows of $A$ indexed by $j_1,\ldots,j_{n-d}$ and all $n-d$ columns. We thus obtain
\[(\wedge^{n-d}\Phi)(e_\sigma)=\bar{f}_\sigma,\quad\text{and}\quad
(\wedge^{n-d}A)(E_1\wedge\cdots\wedge E_{n-d})=
\sum_{\sigma\in\binom{[r]}{n-d}}[A]_\sigma e_\sigma.\]
Without loss of generality, suppose $\bar{f}_1,\ldots,\bar{f}_{n-d}$ form a basis of $V$. 
Write $\sigma_0=\{1,\ldots,n-d\}$. The vector space $\wedge^{n-d}V$ is one-dimensional and it is spanned
by $\bar{f}_{\sigma_0}$. Thus, for any $\sigma\in\binom{[r]}{n-d}$, there is a unique
$\beta_\sigma\in\kappa$ such that $\bar{f}_\sigma=\beta_\sigma\bar{f}_{\sigma_0}$.

In conclusion, $\wedge^{n-d}(\Phi A)$ is non-zero if and only if $\sum_{\sigma}[A]_{\sigma}\beta_\sigma\neq0$.
This determines a non-trivial polynomial condition on the $b_{i,j}$. 
By imposing this condition at the generic point of each irreducible component of $X$, we ensure that the desired additional genericity conditions hold for the chosen linear combinations, thereby obtaining $F_1,\ldots,F_{n-d}$.
\end{proof}

\begin{rke}
Let $\mathbb A^N_{\K}$ denote the affine space parametrizing the coefficients
$b=(b_{i,j})$ used in the construction of $Z$. Several results in the paper
require $b$ to satisfy additional genericity conditions beyond those used in
the prime avoidance construction of $Z$. These conditions depend on the application.
For instance,  when $X$ is reduced, Proposition~\ref{prop:reduced_construction} imposes
conditions guaranteeing that $Z$ is generically reduced along $X$. Further genericity conditions
are used in Proposition~\ref{jacobian conservation}~\rm{(ii)} to ensure that the Jacobian
ideals of $X$ and $Z$ have the same integral closure locally at the singular points
of $X$. Proposition~\ref{jacobian transversality} requires passing to a larger parameter
space associated with the construction of $\tilde Z$; genericity in these additional
parameters yields local transversality of $X$ and the residual curve $W$. Since each of these conditions is given by the non-vanishing of a polynomial in the corresponding parameters, they may be imposed simultaneously by choosing those
parameters sufficiently general.
\end{rke}

\subsection{Computing the complete intersection discrepancy}
Throughout the rest of the paper, $X\subset\Pj^n_\K$ is a Cohen--Macaulay curve that is generically a complete intersection. We fix a complete intersection curve $Z$ as constructed in Proposition~\ref{prop:generic_CI_construction}; when $X$ is reduced, we adopt the more explicit construction of $Z$ from Proposition~\ref{prop:reduced_construction}.  Denote by $\mathrm{Jac}(X)$ the Jacobian ideal of $X$ in $S(X)$, which is the first Fitting ideal of $\Omega_{S(X)/\K}^1$. Denote by $\mathrm{Jac}(Z)$ the image of the Jacobian ideal of $Z$ in $S(X)$. Set $\mathrm{Jac}(X,x):=\mathrm{Jac}(X)\mathcal{O}_{X,x}$ and $\mathrm{Jac}(Z,x):=\mathrm{Jac}(Z)\mathcal{O}_{X,x}$.
Set $$e(\mathrm{Jac}(X)):=\sum_{x \in X_{\mathrm{sing}}}e(\mathrm{Jac}(X,x)) \ \ \text{and} \ \ e(\mathrm{Jac}(Z)):=\sum_{x \in Z_{\mathrm{sing}}\cap X}e(\mathrm{Jac}(Z,x))$$
where $e(\mathrm{Jac}(Z,x))$ and $e(\mathrm{Jac}(X,x))$ are the corresponding Hilbert--Samuel multiplicities of the ideals $\mathrm{Jac}(Z,x)$ and $\mathrm{Jac}(X,x)$.

We say that $X$ is {\it smoothable} if there exists a flat morphism $s: \mathcal{X} \rightarrow T$ with $T$ an affine irreducible smooth curve over $\K$ and $\mathcal{X} \subset \mathbb{P}_{\K}^n \times T$ such that $\mathcal{X}_{t_0}=X$ for a closed point $t_0 \in T$ and $\mathcal{X}_t$ is smooth for all $t \in T$ with $t \neq t_0$. We will show below that $s$ induces a deformation of any choice of equations for $X \subset \mathbb{P}_{\K}^n$ which in turn induces an embedded deformation $\mathcal{Z} \rightarrow T$ of $\mathcal{Z}_{t_0}=Z$.  Recall that $X \subset \mathbb{P}_{\K}^n$ is an {\it almost complete intersection} (aci) curve if $X$ can be defined as the zero locus of $n$ homogeneous polynomials $f_1,\ldots,f_n$.

\begin{prop}\label{jacobian conservation}
The following holds:
\begin{itemize}
\item [\rm{(i)}] Assume $X$ is smooth. Let $h \in \mathcal{O}_{\mathbb{P}_{\K}}^n(1)$  such that $\mathbb{V}(h)$ avoids $Z_{\mathrm{sing}}$.
Then
$$I(X,W)=\dim_{\K} S(X)_{(h)}/\Jac(Z).$$
\item [\rm{(ii)}] Assume $X$ is a reduced local complete intersection. Then 
$$
I(X,W)= e(\mathrm{Jac}(Z))-e(\mathrm{Jac}(X)).
$$
If $(b_{i,j})$ and $\ell_i$ are general, then $$I(X,W)=\sum_{x \in Z_{\mathrm{sing}} \cap X_{\mathrm{sm}}} \dim_{\K}\mathcal{O}_{X,x}/\mathrm{Jac}(Z,x).$$
\item [\rm{(iii)}] Assume $X$ is reduced and smoothable. Then 
$$
I(X,W)= \dim_{\K} S(\mathcal{X}_t)_{(h)}/\Jac(\mathcal{Z}_t)
$$
for $t \neq t_0$  in an affine neighborhood of $t_0$.
\item [\rm{(iv)}] Assume $X$ is a Cohen--Macaulay almost complete intersection curve that is generically a complete intersection. Then 
$$
I(X,W)= \Pi_{i=1}^n d_i-d_n\mathrm{deg}(X).
$$
\end{itemize}
\end{prop}
\begin{proof}
Consider $\rm{(i)}$. Observe that $$\dim_{\K} S(X)_{(h)}/\Jac(Z)=\sum_{x \in X \cap W}
\dim_{\K}\mathcal{O}_{X,x}/\mathrm{Jac}(Z,x).$$ By definition
$I(X,W)=\sum_{x \in X \cap W}I_{x}(X,W)$. By Corollary \ref{smooth curve} we have
$I_{x}(X,W)=\dim_{\K}\mathcal{O}_{X,x}/\mathrm{Jac}(Z,x)$.  By combining the last three
identities we obtain the desired result. 

Consider $\rm{(ii)}$. Suppose $x \in X_{\mathrm{sing}}$. Then
$I_{x}(X,W)=e(\mathrm{Jac}(Z,x))-e(\mathrm{Jac}(X,x))$ by Corollary \ref{lci}. Suppose
$x \in X_{\mathrm{sm}}$. Then $e(\mathrm{Jac}(Z,x))=
\dim_{\K}\mathcal{O}_{X,x}/\mathrm{Jac}(Z,x)$ and $e(\mathrm{Jac}(X,x))=0$.
Thus by Corollary \ref{smooth curve} $I_{x}(X,W)=e(\mathrm{Jac}(Z,x))$.
If the $\ell_i$ are chosen so that $\mathbb{V}(\ell_i) \cap X_{\mathrm{sing}} = \emptyset$ and
the $(b_{i,j})$ are general so that $\mathrm{Jac}(Z,x)$ and $\mathrm{Jac}(X,x)$ have the same
integral closure for each $x \in X_{\mathrm{sing}}$ (see \cite[Proposition 2.3]{BGR}), then
$e(\mathrm{Jac}(Z,x))-e(\mathrm{Jac}(X,x))=0$ for $x \in X_{\mathrm{sing}}$ which proves \rm{(ii)}.

Consider $\rm{(iii)}$. Note that $Z$ is constructed using a particular choice of equations for $X \subset \mathbb{P}_{\K}^n$. We want to show that the smoothing $s: \mathcal{X} \rightarrow T$ induces a deformation of this particular choice of equations of $X$, and therefore $s$ gives rise to an embedded deformation $\mathcal{Z} \rightarrow T$ of $Z$ such that $\mathcal{Z}_t$ is generically reduced along $\mathcal{X}_t$.  

Assume that $T\subset \mathbb A^l_\K$ and, after a linear change of coordinates,
that $t_0=0$. Let $g_1(\mathbf x,\mathbf y),\ldots,g_m(\mathbf x,\mathbf y)$
be defining equations for $\mathcal X\subset \mathbb P^n_\K\times \mathbb A^l_\K$, where $\mathbf x$ and $\mathbf y$ denote the projective and affine coordinates,
respectively. As before, let
$f_1(\mathbf x),\ldots,f_r(\mathbf x)$
be homogeneous generators of the ideal of $X\subset\mathbb P^n_\K$. For $i=1,\ldots,m$, let
$g_i(\mathbf x):=g_i(\mathbf x,\mathbf y)\big|_{\mathbf y=0}.$
Since the polynomials $g_1(\mathbf x),\ldots,g_m(\mathbf x)$ generate the ideal of $X$,
for each $i=1,\ldots,r$ there exist polynomials $r_{i,j}(\mathbf x)$ such that
$f_i(\mathbf x)=\sum_{j=1}^m r_{i,j}(\mathbf x)\,g_j(\mathbf x).$
Choose lifts $r_{i,j}(\mathbf x,\mathbf y)\in \K[\mathbf x]\otimes_\K \K[\mathbf y]$
satisfying
$r_{i,j}(\mathbf x,\mathbf y)\big|_{\mathbf y=0}=r_{i,j}(\mathbf x),$
and define
$f_i(\mathbf x,\mathbf y):=
\sum_{j=1}^m r_{i,j}(\mathbf x,\mathbf y)\,g_j(\mathbf x,\mathbf y).$ Let
$I_1=\langle g_1(\mathbf x,\mathbf y),\ldots,g_m(\mathbf x,\mathbf y)\rangle$
and 
$I_2=\langle f_1(\mathbf x,\mathbf y),\ldots,f_r(\mathbf x,\mathbf y)\rangle.$
By construction, $I_2\subseteq I_1$. Identifying these ideals with their images in
$\mathcal O_{\mathbb P^n_\K\times T}$, our goal is to show that, after possibly replacing
$T$ by a smaller affine neighbourhood of $0$, one has $I_1=I_2$.

First, replace $T$ by an affine neighborhood of $0$ if necessary so that $\mathbb{P}_{\K}^n \times \{0\}$ is a principal Cartier divisor in $\mathbb{P}_{\K}^n \times T$ defined by the vanishing of some $t \in \mathcal{O}_T$. Because of flatness $t$ is a nonzerodivisor of $\mathcal{O}_{\mathbb{P}_{\K}^n \times T}/I_1$. Thus the image of $I_2$ in $\mathcal{O}_{\mathbb{P}_{\K}^n}$ is the same as the image of $I_2$ in $I_1/tI_1$. But the images of $I_1$ and $I_2$ in $\mathcal{O}_{\mathbb{P}_{\K}^n}$ are equal. Thus $(I_1/I_2) /t(I_1/I_2)=0$. Set $V:=\mathrm{Supp}_{\mathbb{P}_{\K}^n \times T}(I_1/I_2)$. Because $I_1/I_2$ is a coherent $\mathcal{O}_{\mathbb{P}_{\K}^n \times T}$-module, $V$ is closed in $\mathbb{P}_{\K}^n \times T$. Consider the proper morphism $\pi:\mathbb{P}_{\K}^n \times T \rightarrow T$. Then $\pi (V)$ is a closed subset in $T$. 
But $V$ avoids $\mathbb{P}_{\K}^n \times \{0\}$.
Thus $\pi (V)$ does not contain $t_0=0$ and so $\pi (V)$ is a finite set of points. Therefore, by replacing $T$ by an affine neighborhood of $t_0$ we have $I_1=I_2$. So we can assume that $\mathcal{X}$ is defined in $\mathbb{P}_{\K}^n \times T$ by $f_1({\bf x,y}), \ldots, f_r({\bf x,y})$. 

Define $\mathcal Z\subset \mathbb P_k^n\times T$ by replacing, in the
definition of $Z$ in Section \ref{construction}, $f_i(x)$ by $f_i(x,y)$ for
$i=1,\ldots,r$. Clearly, $\mathcal Z_{t_0}=Z$. Since
$\dim Z=1$, by upper semicontinuity of fiber dimension, after replacing
$T$ by an affine neighborhood of $t_0$, we may assume that
$\dim \mathcal Z_t\leq 1$ for every $t\in T$. Since $\mathcal X_t\subset \mathcal Z_t$ and
$\mathcal X_t$ is a curve, we have $\dim\mathcal Z_t=1$ for every $t\in T$. On the other hand, $\mathcal Z$ is cut out by $n-1$ equations in $\mathbb P_k^n\times T$, so by Krull's height theorem each irreducible component of $\mathcal Z$ has dimension at least
$2$. No irreducible component of $\mathcal Z$ can be contained in a fiber of $\mathcal Z\to T$, since every fiber is one-dimensional. It follows that every irreducible component of $\mathcal Z$ has
dimension $2$. Hence $\mathcal Z$ is a complete intersection in $\mathbb P_k^n\times T$, and in particular is Cohen--Macaulay. Since every irreducible component of $\mathcal Z$ dominates $T$, a local parameter of the smooth curve $T$ is contained in no minimal prime of $\mathcal Z$. As $\mathcal Z$ is Cohen--Macaulay, it has no
embedded associated primes, and hence such a local parameter is a nonzerodivisor on $\mathcal O_{\mathcal Z}$. Therefore, $\mathcal Z\to T$ is flat.

Let $\mathcal J\subset \mathcal O_{\mathcal X}$ be the ideal generated by
the $(n-1)\times(n-1)$ minors of the relative Jacobian matrix
$J_{\mathcal Z/T}$, and set
$\mathcal B:=V(\mathcal J)\subset \mathcal X.$
Since $\mathcal{Z}_{t_0}=Z$ is generically reduced along
$\mathcal{X}_{t_0}=X$, the Jacobian matrix of $\mathcal{Z}_{t_0}$ has maximal rank
$n-1$ at the generic point of each irreducible component of $X$.
Consequently, $\mathcal B_{t_0}$ contains no irreducible component
of $\mathcal{X}_{t_0}$, and hence
$\dim \mathcal B_{t_0}\leq 0.$
Since $\mathcal X\to T$ is projective, the morphism
$\mathcal B\to T$ is proper. By upper semicontinuity of fiber
dimension, after replacing $T$ by an affine neighborhood of $t_0$,
we may assume that
$\dim \mathcal B_t\leq 0
$
for every $t\in T$. Thus $\mathcal B_t$ contains no irreducible
component of $\mathcal{X}_t$. It follows that the Jacobian matrix of $\mathcal{Z}_t$
has maximal rank at the generic point of every irreducible component
of $\mathcal{X}_t$. By the Jacobian criterion, $\mathcal{Z}_t$ is smooth, hence reduced, at each of these generic points. Since $\mathcal{X}_t\subset \mathcal{Z}_t$, the induced surjection of the corresponding zero-dimensional local rings is therefore an isomorphism.  Hence $\mathcal Z_t$ is generically reduced along $\mathcal X_t$ for every $t\in T$.

Let $\Sigma(\mathcal Z/T)$ denote the relative singular locus of $\mathcal Z\to T$. Choose a linear form $h$ such that $V(h)\cap Z_{\mathrm{sing}}=\varnothing$. The closed subset $V(h)\cap\Sigma(\mathcal Z/T)\subset\mathcal Z$
is proper over $T$, and its fiber over $t_0$ is empty. Its image in $T$ is therefore a closed subset not containing $t_0$. After shrinking $T$ once more, we may thus assume that $V(h)\cap(\mathcal Z_t)_{\mathrm{sing}}=\varnothing$
for every $t\in T$.

Set $\mathcal{W}_t:=\overline{\mathcal{Z}_t\setminus \mathcal{X}_t}$. By Corollary \ref{const. cid} we have $I(X,W)=I(\mathcal{X}_{t_0},\mathcal{W}_{t_0})=I(\mathcal{X}_t,\mathcal{W}_t)$ for $t \in T$. By Proposition \ref{jacobian conservation} \rm{(i)} $I(\mathcal{X}_t,\mathcal{W}_t)= \dim_{\K} S(\mathcal{X}_t)_{(h)}/\Jac(\mathcal{Z}_t)$. The proof of  $\rm{(iii)}$ is now complete. 


Consider $\rm{(iv)}$. By construction in Section \ref{construction}, since $X$ is an aci curve we have $I_X=(I_Z,f_n)$. Set $l:=c_{1}(\mathcal{O}_{\mathbb{P}^n}(1))$. 
Since $X$ is Cohen–Macaulay and $X$ and $W$ are geometrically linked by the complete intersection $Z$, the residual curve $W$ is Cohen–Macaulay. Moreover, $W$ shares no irreducible component with $X = \mathbb{V}(I_Z,f_n)$; hence $f_n$ is a nonzerodivisor on $W$. Thus  $I(X,W)=V(f_n)\cdot[W]=d_n l\cdot[W]=d_n\deg{W}$, where $[W]$ is the fundamental cycle of $W$. But $$\deg(W)=\deg(Z)-\deg(X)=\Pi_{t=1}^{n-1} d_i-\deg(X).$$ Finally, $I(X,W)=\Pi_{i=1}^n d_i - d_n \deg(X)$ as desired. 

Alternatively, let $h \in \mathcal{O}_{\mathbb{P}^n}(1)$ such that $H=\mathbb{V}(h)$ does not contain an irreducible component of $Z$. 
Consider the trivial one-parameter family $W \times  \mathbb{A}_{\K}^1 \rightarrow  \mathbb{A}_{\K}^1$ with $\mathbb{A}_{\K}^1:=\mathrm{Spec}(\K[t])$.
Set $D_n:=\mathbb{V}(f_n+th^{d_n}) \subset \mathbb{P}_{\K}^n \times \mathbb{A}_{\K}^1$. For each closed point $t' \in \mathbb{A}_{\K}^1$ denote by $D_n(t')$ the fiber of  $D_n$ over $t'$. After possibly replacing $\mathbb{A}_{\K}^1$ by a small affine neighborhood $T$ of $0$ we can assume that $D_n \cap (W \times T)$ is Cohen--Macaulay. 

Consider the $1$-cycle $D_n\cdot[W\times T]$ in $A_{1}(W\times T)$ ($1$-cycles modulo rational equivalence). By conservation of number \cite[Prop.\ 10.2]{Ful} applied to  $D_n\cdot[W\times T]$ and the proper map $W\times T \rightarrow T$ we obtain the following equality of degrees 
\begin{equation}\label{moving}
\int D_n(t')\cdot[W]=\int D_n(0)\cdot[W]=I(X,W)
\end{equation}
for $t' \neq 0$. Next, consider the equality of fundamental cycles $[Z]=[X]+[W]$. For $t' \neq 0$ general, intersecting with $D_n(t')$ gives
\begin{equation}\label{Cartier-fund.}
D_n(t')\cdot[Z]=D_n(t')\cdot[X]+D_n(t')\cdot[W]
\end{equation}
By definition, $D_n(t') \cdot [X]$ is a $0$-cycle with $\int D_n(t') \cdot [X] = d_n \, \deg(X).$
Since $D_n(t') \cap Z$ is a zero-dimensional complete intersection of type $(d_1, \ldots, d_n)$, Bezout's theorem gives $\int D_n(t')\cdot[Z]=\Pi_{i=1}d_i$. Combining this with (\ref{Cartier-fund.}) and (\ref{moving}) yields
$$\Pi_{i=1}^n d_i = d_n \deg(X) + I(X,W)$$
proving $\rm{(iv)}$. This second approach to proving (\ref{aci genus}) can be adapted to the local setup as shown below in Proposition \ref{LGT-aci}.
\end{proof}

\begin{ex}\label{twisted cubic}
Let us illustrate  Proposition \ref{jacobian conservation} \rm{(i)} and \rm{(iv)} with the twisted cubic. This is a classic example of a
space curve that is not a complete intersection. It is cut out from  $\Pj^3_\K$ by the equations
\[f_1=x_{1}x_{3}-x_{2}^2,\:f_2=x_{0}x_{3}-x_{1}x_{2},\:f_3=x_{0}x_{2}-x_{1}^2.\]

Any complete intersection curve $Z$ satisfying Proposition~\ref{prop:reduced_construction} is defined by two quadrics. Therefore, by (\ref{eq:arithmetic genus}), the intersection multiplicity $I(X,W)$ should be independent of the choice of $Z$.

A general choice of $Z$ satisfying Proposition~\ref{prop:reduced_construction} is given by $Z=\mathbb{V}(f_1,2f_2+f_3)$. A direct computation for
the image of $\mathrm{Jac}(Z)$ in $S(X)_{(x_0)}$ gives $$\mathrm{Jac}(Z)=
(x_2-2x_3,x_1-4x_3,8x_{3}^2-x_0x_3).$$ So the subscheme in $\Pj^3_\K$ defined by
$\mathrm{Jac}(Z)$ is the two reduced points $P_1=[1:0:0:0]$ and $P_2=[8:4:2:1]$.
Thus $I(X,W)=2$. In fact, through a direct ideal quotient computation via \textsc{Singular} \cite{DGPS} one finds
$W=\mathbb{V}(x_1-4x_3, x_2-2x_3)$. Furthermore, this line intersects $X$ locally
transversally in the two points $P_1$ and $P_2$, a phenomenon that will be generalized in Proposition~\ref{jacobian transversality}.

Consider the family of complete intersections $Z_a=\mathbb{V}(f_1,sf_2+tf_3)$ with
$a:=[s:t]\in\Pj^1_\K$. Set $W_a=\overline{Z_a\setminus X}$. A direct computation shows that there
are two distinct points of intersection $[1:0:0:0]$ and $[t^3:t^2s:ts^2:s^3]$ unless $a=[1:0]$.
For $a=[1:0]$ we have $\mathrm{Jac}(Z_{[1:0]})=I_X+I_{W_{[1:0]}}=(x_{1}^2,x_2,x_3)$. So locally at
$[1:0:0:0]$ the two curves $X$ and $W_{[1:0]}$ are tangent. As expected by Proposition
\ref{jacobian conservation} \rm{(i)} for each $a$ we have $I(X,W_a)=2$.

Because $X$ is an aci curve of degree $3$ in  $\Pj^3_\K$ with $d_1=d_2=d_3=2$, Proposition \ref{jacobian conservation} \rm{(iv)} yields $I(X,W)=2^3-2\cdot3=2$ in agreement with the computations above. 

For certain natural choices of $Z$, the residual curve can be described explicitly using liaison theory. Since $X$ is a codimension $2$ Cohen--Macaulay curve, the Hilbert--Burch theorem provides a $3\times2$ presentation matrix for its ideal. The maximal minors of this matrix generate $I_X$, while the Apéry--Gaeta theorem \cite[Prop.~21.24]{Ei} shows that any two of these minors define a complete intersection $Z$ containing $X$, and that the entries of the remaining row generate the ideal of the corresponding residual curve $W$. Indeed, the ideal of $X$ is generated by the maximal minors of the matrix
\[
M=\begin{pmatrix}
x_0 & x_1 \\
x_1 & x_2 \\
x_2 & x_3
\end{pmatrix}.
\]
Each $f_i$ is the $2\times2$ minor of $M$ obtained by omitting the $i$-th row. Each of the pairs
$(f_1,f_2)$, $(f_1,f_3)$ and $(f_2,f_3)$ forms a regular sequence in $\K[x_0,x_1,x_2,x_3]$.
The corresponding residual curves are given by the ideals
$(x_2,x_3)$, $(x_1,x_2)$ and $(x_0,x_1)$, respectively. The sum ideals $I_X+I_W$ are then
$(x_1^2,x_2,x_3)$, $(x_0x_3,x_1,x_2)$ and $(x_0,x_1,x_2^3)$, respectively. They correspond to the
divisors $2[1:0:0:0]$, $[1:0:0:0]+[0:0:0:1]$, and $2[0:0:0:1]$. In all three cases,
we find $I(X,W)=2$ as expected.
\end{ex}
\vspace{.2cm}
\begin{center}
{\textbf{Proof of Corollary \ref{genus-degree}}}
\end{center}
\vspace{.2cm}
 Clearly, \eqref{aci genus} follows directly from \eqref{eq:arithmetic genus} together with Proposition~\ref{jacobian conservation}\,(iv).

Since $X$ is projective, reduced and connected over an algebraically closed field, by 
\cite[\href{https://stacks.math.columbia.edu/tag/0BUG}{Tag 0BUG}]{St} we have $H^{0}(X,\mathcal{O}_X)=\K$. Consequently, 
$$p_a(X)=1-\chi(X,\mathcal{O}_X)=\dim_{\K}H^{1}(X,\mathcal{O}_X) \geq 0.$$
Since $X$ is reduced, it is Cohen–Macaulay and generically a complete intersection, so (\ref{aci genus}) gives
$$\deg(X) \geq \frac{\Pi_{i=1}^{n}d_i-2}{\sum_{i=1}^{n}d_i-n-1}.$$


Because $d_i \geq 2$ ($X$ is nondegenerate), the numerator and the denominator of the last fraction are positive. Pick an index $j \in \{1, \ldots, n \}$. Fix the degrees $d_i$ with $i \neq j$. Set $A_j:=\Pi_{i \neq j}d_i$ and $B_j:=\sum_{i \neq j}d_i$. We want to show that the function $q(d_j):=\frac{d_jA_j-2}{d_j+B_j-n-1}$ is monotonic in $d_j$:
$$q(d_j+1)=\frac{d_jA_j-2+A_j}{d_j+B_j-n-1+1} \geq \frac{d_jA_j-2}{d_j+B_j-n-1}=q(d_j)$$
which after cross-multiplication is equivalent to showing that 
$$ A_jB_j+2 \geq A_j(n+1).$$
But $B_j=\sum_{i \neq j}d_i \geq 2(n-1) \geq n+1$ because $d_i \geq 2$ and $n \geq 3$. Thus the function $Q(d_1, \ldots, d_n):=\frac{\Pi_{i=1}^{n}d_i-2}{\sum_{i=1}^{n}d_i-n-1}$ is monotonic in each variable, and so it is bounded from below by $Q(d_n, \ldots, d_n)$ proving (\ref{degree ineq}). \hfill\qed

The inequality of Corollary \ref{genus-degree} is sharp. Equality is attained by the twisted cubic curve of Example \ref{twisted cubic}. It is a rational curve of degree $3$, defined by three quadrics.

Next we show how to compute $I(X,W)$ when $X$ has locally
smoothable singularities or when $X$ is an almost complete intersection. We say that $X$ is {\it locally smoothable} at  $x$ if there exists an affine neighborhood $(X,x) \subset \mathbb{A}_{\K}^n$ of $x$ in $X$ and a flat morphism $(\mathcal{X},x) \rightarrow T$  with $T$ an affine irreducible smooth curve, such that $(\mathcal{X},x) \subset \mathbb{A}_{\K}^n \times T$, $\mathcal{X}_{t_0}=(X,x)$ for a closed point $t_0 \in T$ and $\mathcal{X}_{t}$ is smooth for $t \neq t_0$. 

Note that $(\mathcal{X},x) \rightarrow T$ induces a deformation $(\mathcal{Z},x) \rightarrow T$ of an affine neighborhood $(Z,x)$ of $x$ in $Z$. Set $\mathcal{W}:=\overline{\mathcal{Z}\setminus \mathcal{X}}$ and set $\mathcal{S}:=\mathcal{X} \times_{\mathcal{Z}} \mathcal{W}$. Then $\mathcal{S} \rightarrow T$ is quasi-finite. By \cite[\href{https://stacks.math.columbia.edu/tag/02LK}{Tag 02LK}]{St} (``\'Etale localization of quasi-finite morphisms'') there exists an elementary étale neighborhood $(T',t) \rightarrow (T,t_0)$ such that after the corresponding base change 
\begin{equation}\label{etale}
\begin{tikzcd}
(\mathcal{S'},x') \arrow[r, hook] \arrow[dr] &
(\mathcal{X}',x') \arrow[r] \arrow[d]
& (\mathcal{X},x) \arrow[d] \\
& (T',t_{0}') \arrow[r]
&  (T,t_0)\\
\end{tikzcd}
\end{equation}
we obtain a finite morphism $(\mathcal{S}',x') \rightarrow (T',t_{0}')$. Note that $(\mathcal{Z}',x'):=(\mathcal{Z},x) \times_{T} T'$ is a complete intersection in $\mathbb{A}_{\K}^n \times T'$.
Write $X_{\mathrm{sm}}$ for the smooth part of $X$ and
$X_{\mathrm{sing}}$ for the singular locus of $X$. We have
$$I(X,W) = \sum_{x \in X_{\mathrm{sm}} \cap W} I_{x}(X,W)+
\sum_{x \in X_{\mathrm{sing}}\cap W}I_{x}(X,W).$$
Let $h$ be a linear form on $\mathbb{P}_{\K}^n$ such that $\mathbb{V}(h)$ avoids $Z_{\mathrm{sing}}$. In particular,
$\mathbb{V}(h)$ avoids
$X\cap W$ and $X_{\mathrm{sing}}$. Set $X_{(h)}:=X \cap D_{+}(h)$. Denote by $A(X_{(h)})$ the affine coordinate ring of
$X_{(h)}\subset \mathbb{A}_{\K}^n$.
Let $g \in I_{X_{\mathrm{sing}}}$ such that $g \not \in I_{X_{\mathrm{sm}}\cap W}$.

\begin{prop}\label{smoothable sing}
Suppose $X \subset \mathbb{P}_{\K}^n$ is a reduced curve with locally smoothable singularities. Consider the étale base change (\ref{etale}).  For $t' \in T'$ with $t' \neq t_{0}'$, identify the Jacobian ideal
$\mathrm{Jac}(\mathcal{Z}_{t}')$ with its image in $\mathcal{O}_{\mathcal{X}_{t}'}$. Then for each $x \in X_{\mathrm{sing}}$ and $t' \neq t_{0}'$ we have

$$I_{x}(X,W)=\sum_{x_{t'}' \in \mathcal{X}_{t'}' \cap \mathcal{W}_{t'}'}\dim_{\K} \mathcal{O}_{\mathcal{X}_{t'}',x_{t'}' }/\mathrm{Jac}(\mathcal{Z}_{t'}').$$
Also,
$$\sum_{x \in X_{\mathrm{sm}} \cap W} I_{x}(X,W)=\dim_{\K}A(X_{(h)})_{g}/\mathrm{Jac}(Z).$$
\end{prop}
\begin{proof} Suppose $x \in X_{\mathrm{sing}}$. Because $\K$ is algebraically closed and $T$ is of finite type over $\K$ we have an equality of residue fields $\K=\kappa(t_{0})=\kappa(t_{0}').$ Thus $\mathcal{Z}_{t_0}=\mathcal{Z}_{t_0}'$, $\mathcal{X}_{t_0}=\mathcal{X}_{t_0}'$ and $\mathcal{W}_{t_0}=\mathcal{W}_{t_0}'$. Therefore, $I_{x}(X,W)=I_{x'}(\mathcal{X}_{t_{0}'},\mathcal{W}_{t_{0}'})$.
We will use \cite[Theorem 3.1]{BGR}. Its proof is entirely algebraic, except for a complex analytic argument used to ensure that $S \rightarrow T$ is finite after shrinking $S$ and $T$. In our setup, this step is replaced by a base change to an étale neighborhood of $(T,t_0)$. By \cite[Theorem 3.1]{BGR}
we obtain $I_{x'}(\mathcal{X}_{t_{0}'},\mathcal{W}_{t_{0}'})=\sum_{x_{t'}' \in \mathcal{X}_{t'}' \cap \mathcal{W}_{t'}'}I_{x_{t'}'}(\mathcal{X}_{t'}',\mathcal{W}_{t'}')$. The first identity is obtained by applying Corollary \ref{smooth curve} to each summand. The second identity follows from Proposition \ref{jacobian conservation} \rm{(i)}. 
\end{proof}
Note that when the parametrization of $(X,x)$ is known, we can compute $I_{x}(X,W)$ as 
$$I_{x}(X,W)=e(\Jac(Z,x))-2\delta_x-e(R_X).$$ 

Next, we show how to compute the complete intersection discrepancy directly from the equations of $X$ and $Z$ when $X$ is locally an aci curve. To keep the exposition as simple as possible, we will work over $\K=\mathbb{C}$; the general case can be treated via a suitable étale localization as discussed above.
\begin{prop}\label{LGT-aci}
Suppose $(X,0)=\mathbb{V}(f_1, \ldots, f_n) \subset (\mathbb{C}^n,0)$ is a Cohen--Macaulay aci curve that is generically a complete intersection. Fix one of the defining  equations, say $f_n$. Then there exists a complete intersection $(Z,0) \subset (\mathbb{C}^n,0)$ such that generically $Z=X$ and $$I_{X,0}=(I_{Z,0},f_n).$$ Set $W:=\overline{Z\setminus X}$. Choose a representative $Z_\delta= Z \cap  \mathring{\mathbb{B}}(0,\delta)$ of the germ $(Z,0)$, where $\mathring{\mathbb{B}}(0,\delta)\subset \C^n$ is the open ball of radius $\delta$ centered at $0$, such that $X \cap W = \{0\}$ in $Z_{\delta}$. Then, for $\varepsilon \in \mathbb{C}$ with $0<|\varepsilon|\ll\delta$, one has
$$I_{0}(X,W)=\sum_{x \in Z_{\delta} \cap \{f_n+\varepsilon=0\}}\dim_{\mathbb{\mathbb{C}}} \mathcal{O}_{Z,x}/(f_n+\varepsilon).$$
\end{prop}
\begin{proof}
In Proposition~\ref{prop:generic_CI_construction}, we showed how to find
general linear combinations
\begin{equation}\label{eq:lin_comb}
g_i=\sum_{j=1}^n a_{i,j}f_j,\quad i=1,\ldots,n-1,\:a_{i,j}\in \mathbb{C},
\end{equation}
such that $Z$, defined by $I_Z=(g_1,\ldots,g_{n-1})$, is a complete intersection curve,
that contains $X$ as a union of some of its irreducible components.
We claim that for  any $i\in\{1,\ldots,n\}$, one can impose an additional genericity condition on the coefficients
$a_{i,j}$ in \eqref{eq:lin_comb}, so that
\[I_X=I_Z+(f_i).\]

Up to reindexing, we can assume that $i=n$ and proceed without loss of generality.
Put the relations in \eqref{eq:lin_comb} in a matrix form. Define the column
vectors $[f]:={}^T(f_1,\ldots,f_n)$ and $[g]:={}^T(g_1,\ldots,g_{n-1})$ and the matrix
$A=(a_{i,j})_{\substack{1\leqslant i\leqslant n-1 \\ 1\leqslant j\leqslant n\phantom{-1}}}$,
so that
\begin{equation}\label{eq:matrix_pres}
[g]=A[f].
\end{equation}
Isolate the leftmost square submatrix of $A$ of size $(n-1)\times(n-1)$
\[A=\Big(\:\boxed{A'}\:\Big|\: C'\Big),\]
where $A'\in\mathrm{Mat}(n-1,\mathbb{C})$ is a square matrix and $C'\in\mathbb{C}^{n-1}$
is a column vector. We add the following genericity condition to the coefficients $a_{i,j}$:
$\det(A')\neq0$. Multiply \eqref{eq:matrix_pres} by $A'^{-1}$ to obtain $A'^{-1}[g]=\widetilde A[f]$,
where
\[\widetilde{A}=\Big(\:\boxed{I_{n-1}}\:\Big|\: C\Big),\]
and $C\in \mathbb{C}^{n-1}$. Write $C={}^T(c_1,\ldots,c_{n-1}),\:c_i\in\mathbb{C}$ and
${}^T(h_1,\ldots,h_{n-1})=A'^{-1}[g]$. It is clear that $I_Z=(h_1,\ldots,h_{n-1})$.
Furthermore, $h_i=f_i+c_if_n$ for $i=1,\ldots,n-1$. Thus
$(h_1,\ldots,h_{n-1},f_n)=(f_1,\ldots,f_n)$ proving the claim. For this choice of a complete intersection $Z$ set $W:=\overline{Z\setminus X}$. Fix a representative $Z_\delta := Z \cap \mathring{\mathbb{B}}(0,\delta)$
of the germ \((Z,0)\) such that \(X \cap W = \{0\}\) in \(Z_\delta\). Set
$X_\delta := X \cap \mathring{\mathbb{B}}(0,\delta)$ and $W_\delta := W \cap \mathring{\mathbb{B}}(0,\delta).$

Consider the trivial family \(W_\delta \times (\mathbb{C},0) \to (\mathbb{C},0)\), and let \(t\) be a uniformizing parameter of \(\mathcal{O}_{\mathbb{C},0}\). Set
\[
E_n := \mathbb{V}(f_n + t) \subset \mathbb{C}^n \times \mathbb{C},
\]
and denote by \(E_n(\varepsilon)\) the fiber over \(t = \varepsilon\). Since \(W_\delta \times (\mathbb{C},0)\) is Cohen--Macaulay by \cite[Thm.\ 21.23(b)]{Ei}, it follows that \(E_n \cap (W_\delta \times (\mathbb{C},0))\) is also Cohen--Macaulay. Therefore, by \cite[Thm.\ 18.16(b)]{Ei}, the morphism
$E_n \cap (W_\delta \times (\mathbb{C},0)) \to (\mathbb{C},0)$
is flat. By the Weierstrass Preparation Theorem, this morphism is finite over a sufficiently small disk 
$\mathbb{D} := \mathbb{D}(0,\varepsilon) \subset \mathbb{C}$
with \(0 < |\varepsilon| \ll \delta\).
It follows that
\begin{equation}\label{W-flat}
I_0(X,W)= \dim_{\mathbb{C}} \mathcal{O}_{W,0}/(f_n)= \sum_{x \in W_\delta \cap \{f_n+\varepsilon=0\}} \dim_{\mathbb{C}} \mathcal{O}_{W,x}/(f_n+\varepsilon).
\end{equation}
For \(\varepsilon \neq 0\), we have
\[
E_n(\varepsilon) \cap W_\delta = E_n(\varepsilon) \cap Z_\delta,
\]
since \(E_n(\varepsilon) \cap X_\delta = \emptyset\). Therefore,
\[
I_0(X,W)
=
\sum_{x \in Z_\delta \cap \{f_n+\varepsilon=0\}}
\dim_{\mathbb{C}} \mathcal{O}_{Z,x}/(f_n+\varepsilon)
\]
for \(0 < |\varepsilon| \ll \delta\). This completes the proof.
\end{proof}


\subsection{Transversality}\label{transversality}
Assume that  $\mathrm{char}(\K)=0$ and that $X$ is reduced. We show below that we can further manipulate the equations of $Z$ from Proposition~\ref{prop:reduced_construction} to ensure that $X$ and $W$ are in general position, i.e.\ their points of intersection are ordinary double points in $Z$ when $Z$ is a local complete intersection. We do this using Bertini's and Kleiman's transversality theorems.  

Let $Z=\mathbb{V}(F_1, \ldots, F_{n-1})$ be the complete intersection constructed in Proposition~\ref{prop:reduced_construction}. Choose a new $\ell_{n-1}\in S(\Pj^n)_1$ such that $\mathbb{V}(\ell_{n-1})$ does not
pass through any point $x\in X$ where $\mathrm{rk}(J(Z))<n-1$ and $\ell_{n-1}$
avoids the minimal primes of $(F_1,\ldots,F_{n-1})$. Recall that $I_X=(F_1, \ldots, F_{n-1}, f_n, \ldots, f_r)$. To simplify notation, for $n\leq i \leq r$ set $F_i:=f_i$. For each $i=1, \ldots, n-1$ consider the linear system on
$\Pj_{\K}^n$ spanned by
\begin{equation}\label{linear system}
F_i,\,\ell_{n-1}^{d_i-d_{i+1}}F_{i+1},\,\ldots\,,\ell_{n-1}^{d_i-d_r}F_r.
\end{equation}
Apply successively Bertini's theorem \cite[Cor. 5]{Kl74} starting from $i=1$
to get
\[\tilde{F}_i=\sum_{j=i}^ra_{i,j}\ell_{n-1}^{d_i-d_j}F_j,\:a_{i,j}\in\K,\]
such that $\tilde{Z}:=\mathbb{V}(\tilde{F}_1,\ldots,\tilde{F}_{n-1})$ is smooth  away from the
intersection of the base loci which is
\[\mathbb{V}(\ell_{n-1},\tilde{F}_1,\ldots,\tilde{F}_{n-1})\cup X.\]

But $\mathbb{V}(\ell_{n-1},\tilde{F}_1,\ldots,\tilde{F}_{n-1})$ is a finite set of points. This means
that $\tilde{W}=\overline{\tilde{Z}\setminus X}$ is reduced because $\tilde{Z}$ is a complete intersection and so $\tilde{Z}$ does not have embedded points. By construction
$J(\tilde{Z})|_{\mathbb{V}(\ell_{n-1})\cap X}=J(Z)|_{\mathbb{V}(\ell_{n-1})\cap X}$ is of maximal rank. 
Thus, $\tilde{Z}$ is also reduced along $X$.
Therefore, $\tilde{Z}$ is reduced.

Up to a linear change of coordinates, we can assume $\ell_{n-1}=x_0$. By the Leibniz rule
we have the following identity for Jacobian matrices $J(\tilde{Z})|_X=A^* J(X)$ where
\[A^*=\begin{pmatrix}
 a_{1,1} & a_{1,2}x_0^{d_1-d_2} & a_{1,2}x_0^{d_1-d_3} & \cdots & \cdots & \cdots &
    a_{1,r}x_0^{d_1-d_r} \\
0 & a_{2,2} & a_{2,3}x_0^{d_2-d_3} & \cdots & \cdots & \cdots &
    a_{2,r}x_0^{d_2-d_r} \\
\vdots & \vdots & \ddots & & & \vdots \\
0 & 0 & \cdots & a_{n-1,n-1} & a_{n-1,n}x_0^{d_{n-1}-d_n} & \cdots &
    a_{n-1,r}x_0^{d_{n-1}-d_r}
\end{pmatrix}.\]

By our choice of $x_0$, the locus where $X$ and $\tilde{W}$ intersect is contained in the
principal open $U=D_+(x_0)$. Set $A:=A^*|_{x_0=1}$. 

\begin{prop}\label{jacobian transversality}
Assume $X$ is a reduced local complete intersection curve. 
If $A$ is general, then $X \cap \tilde{W}$ are ordinary double points of $\tilde{Z}$ and 
$$I(X,\tilde{W})= \#(\tilde{Z}_{\mathrm{sing}} \cap X_{\mathrm{sm}}).$$
\end{prop}

\begin{proof}
The set of matrices $A$ is isomorphic to
$\A_{\K}^{(n-1)(2r-n+2)/2}$. We will show 
that there exists a non-empty Zariski open subset of $\A_{\K}^{(n-1)(2r-n+2)/2}$, such that for any matrix with entries from that open subset, $X$ and $\tilde{W}$ intersect locally
transversally, i.e.\ the points of intersection are ordinary double points of $\tilde{Z}$. As the intersection takes place in $U$, we can do this by considering the equations for $X\cap U$, obtained
by dehomogenizing the equations of $X$ along $x_0$, i.e.\ by setting $x_0=1$ in the
equations of $X$. Note that the Jacobian matrix of $X \cap U$ is the restriction of $J(X)$ to $X\cap U$ with its first column deleted. Without loss of generality we will assume that $X \subset \mathbb{A}_{\K}^n$ and $\tilde{Z} \subset \mathbb{A}_{\K}^n$ are affine. The equations for $\tilde{Z}$ in $\mathbb{A}_{\K}^n$ are given by $A(f_{1}', \ldots, f_{r}')^{T}=0$ where $f_{i}'=\tilde{F}_{i}|_{x_0=1}$.

Let $x$ be a singular point in $X$. Because $X$ is a local complete intersection at $x$, by Corollary \ref{lci} $(\tilde{Z},x)=(X,x)$ if and only if $e(\mathrm{Jac}(\tilde{Z},x))=e(\mathrm{Jac}(X,x))$ which is equivalent to $\mathrm{Jac}(\tilde{Z},x)$ and $\mathrm{Jac}(X,x)$ having the same integral closure in $\mathcal{O}_{X,x}$ by Rees' result \cite[Theorem 11.3.1]{SH}. But for a general $A$ the two ideals $\mathrm{Jac}(\tilde{Z},x)$ and $\mathrm{Jac}(X,x)$ have the same integral closure (see \cite[Proposition 2.3]{BGR}). Thus for a general $A$ we can assume that $(\tilde{Z},x)=(X,x)$ for each singular point $x$. So $I(X,\tilde{W})=I(X_{\mathrm{sm}},\tilde{W})$. Therefore, without loss of generality we may assume that $X$ is smooth. 

Denote by $M\subset\Or_{X}^r$ the $\Or_{X}$-module generated by the columns of
$J(X)$. Set $e:=n-1$. Denote by $[M]$ the presentation matrix of $\Or_{X}/M$.
Let $M_e\subset\Or_{X}^e$, the $\Or_{X}$-module generated by the columns of
$[M_e]=A[M]$. Denote by $S$ the subscheme of $X$ defined by
$\Fitt_0(\Or_{X}^e/M_e)$. We claim that $S$ is reduced for general $A$. We have
$[M]\in\mathrm{Mat}(r\times n,\Or_{X})$. Denote by $M^*$ the $\Or_{X}$-submodule
of $\Or_{X}^n$ generated by the columns of $[M]^{\mathrm{tr}}$. We have
$$\Fitt_{r-e}(\Or_{X}^r/M)=\Fitt_{n-e}(\Or_{X}^n/M^*)=\Fitt_{1}(\Or_{X}^n/M^*).$$

Denote by $M'$ the
$\Or_{X}$-submodule of $M^*$ generated by the columns of $[M^*](A)^{\mathrm{tr}}$.
Because $[M^*](A)^{\mathrm{tr}}=(A[M])^{\mathrm{tr}}$ we have
\[\Fitt_0(\Or_{X}^e/M_e)=\Fitt_{1}(\Or_{X}^n/M').\]
Let $x\in S$. Consider the following
short exact sequence
\begin{equation}\label{eq:SES}
\begin{tikzcd}
0 \ar[r]& M^*_x/M'_x \ar[r]& \Or_{X,x}^n/M'_x \ar[r]& \Or_{X,x}^n/M^*_x \ar[r]&0.
\end{tikzcd}
\end{equation}
Because $M$ is the Jacobian module of $X$ and $X$ is smooth, we have $\Or_{X,x}^r=
M_x\oplus F$, where $F$ is an $\Or_{X,x}$-free module of rank $r-e$. Thus
$\Or_{X,x}^n/M^*_x$ is free and so \eqref{eq:SES} is a split exact sequence. By
\cite[\href{http://stacks.math.columbia.edu/tag/07ZA}{Tag 07ZA}]{St} we have
\[\Fitt_{1}(\Or_{X,x}^n/M'_x)=
\sum_{i+j=1}\Fitt_i(M^*_x/M'_x)\Fitt_j(\Or_{X,x}^n/M^*_x)=
\Fitt_0(M^*_x/M'_x),\]
because $\Or_{X,x}^n/M^*_x$ is free of rank $1$, so the only nonzero term in the
sum above is the one corresponding to $(i,j)=(0,1)$. Thus locally at $x$, $S$ is
defined in $\Or_{X,x}$ by $\Fitt_0(\Or_{X,x}^e/(M_e)_x)=
\Fitt_{1}(\Or_{X,x}^n/M'_x)=\Fitt_0(M^*_x/M'_x)$. Note that $M^*$ is a locally free sheaf of rank $e$ on $X$ generated by $r$ global
sections. By Kleiman's transversality theorem \cite[Rmk. 6 and 7]{Kl74} the subscheme
defined by $\Fitt_0(M^*/M')$ is either empty or reduced of dimension $0$. Therefore, if $S$ is nonempty, then $S$ is reduced. Suppose this is the case.

For such a choice of $A$, giving a reduced $S$, note that $[M_e]=J(\tilde{Z})|_{X}$. Then $M_e$ is the Jacobian module $J(\tilde{Z})$ of $\tilde{Z}$ restricted to
$X$. Let $x\in S$. We have $\Fitt_0(\Or_{X}^e/M_e)\otimes\Or_{X,x}=(u)$, 
where $u$ is a uniformizing parameter for $\mathcal{O}_{X,x}$. Without loss of generality  assume that $(\tilde{Z},x)=\mathbb{V}(f_1', \ldots, f_{n-1}')$ where
$(X,x)=\mathbb{V}(f_1', \ldots, f_r')$. Because $\mathcal{O}_{X,x}$ is a PID, locally at $x$, the matrix $[J(\tilde{Z},x)]$ has a Smith normal form with invariant factors $u^{a_1}, \ldots, u^{a_v}$. The matrix $[J(\tilde{Z},x)]$ is of size $(n-1)\times n$. The ideal of $(n-1)\times (n-1)$ minors of $[J(\tilde{Z},x)]$ is the same as that of its Smith normal form. Because $\Fitt_0(\Or_{X,x}^e/J(\tilde{Z},x))=(u)$ we get $v=n-1$ and all the exponents $a_i$ vanish but one, which is equal to $1$. Thus the rank of $[J(\tilde{Z},x)]$ is $n-2$. Without loss of generality we can assume that the Jacobian matrix of the variety $\mathbb{V}(f_2', \ldots, f_{n-1}')$ has the maximal possible rank $n-2$ at $x$. So $C:=\mathbb{V}(f_2', \ldots, f_{n-1}')$ is a smooth surface at $x$. After passing to the completion of $(C,x)$ we can assume that $(\tilde{Z},x)$ is contained in $(\mathbb{A}_{\K}^{2},x)$.
Because $X$ is smooth at $x$ we can assume that $(u,t)$ are local coordinates on
$(\mathbb{A}_{\K}^{2},x)$ with $(X,x)=\mathbb{V}(t)$.
Then $(\tilde{Z},x)=\mathbb{V}(tg(u,t))$ where $\mathbb{V}(g)=(\tilde{W},x)$. But then
\[\mathrm{Fitt}_{0}(\mathcal{O}_{X,x}^{e}/J(\tilde{Z},x))=
(\partial{tg(u,t)}/ \partial{u},\partial{tg(u,t)}/ \partial{t})_{|t=0}=(g(u,0)).\]
But $\Fitt_0(\Or_{X,x}^e/J(\tilde{Z},x))=(u)$, so
$\mathrm{ord}_{u}(g(u,0))=1$. This implies that the linear term of $g(u,t)$ is non-zero, \ie
$\deg(\mathrm{in}(g(u,t)))=1$. This means that the Zariski tangent space of $(\tilde{W},x)$ at $x$
is a line and so $(\tilde{W},x)$ is smooth at $x$.
Also $\dim_\mathbb{\K}\mathcal{O}_{\mathbb{A}^2,x}/(t,g(u,t))=1$. Hence $I_x(X,\tilde{W})=1$.
By \cite[Proposition 3]{Hir57} the two curves $(X,x)$ and $(\tilde{W},x)$ intersect in
$(\mathbb{A}_{\K}^{2},x)$ transversally at $x$.
Thus, for a general $A$, each point in $X\cap \tilde{W}$ is an ordinary double point of $\tilde{Z}$ which contributes exactly $1$ to $I(X,\tilde{W})$. The number of such points is the number of singular points of $\tilde{Z}$ that lie in $X$.
\end{proof}

\begin{rke}
By considering larger linear systems on $\Pj^n_{\K}$ spanned by
\[
f_i,\ h_{i,i+1}f_{i+1},\ldots,h_{i,r}f_r,
\]
where $h_{i,j}=\sum_{|k|=d_i-d_j} a_{ijk}{\bf x}^k,$
one can also derive Proposition~\ref{jacobian transversality} from
\cite[Theorem~4.4(f)]{CU}, without assuming $\operatorname{char}(\K)=0$.
The idea in \cite{CU} is to work over the universal coefficient field, with the
$a_{ijk}$ replaced by independent variables, and then, through a sequence of
reduction steps, express the relevant intersection locus as the determinant of a
generic square matrix. Since the
generic determinant hypersurface is smooth in codimension $3$, one obtains the
desired generic transversality after specializing the coefficients to general
constants in $\K$, using a specialization result of Jouanolou.
\end{rke}

\begin{ex} Assume $\mathrm{char}(\K)=0$. Let $X \subset \Pj_\K^n$ be the rational normal curve. It is a projective curve of degree $n$, genus zero, and it is cut out by the quadric equations $f_{i,j}(x_0, \ldots, x_n)=x_{i}x_{j}-x_{i+1}x_{j-1}$. Because the degrees of the defining equations of $X$ are the same, the constructions of $Z$ and $\tilde{Z}$ are the same: it is enough to consider $n-1$ general $\K$-linear combinations of the $f_{i,j}$. By (\ref{eq:arithmetic genus}) we have $I(X,W)=n(n-3)+2$ (we will show elsewhere that in fact (\ref{eq:cid-jacobian}) can be used to compute combinatorially $I(X,W)$ for smooth monomial curves). Because $\mathrm{deg}(Z)=2^{n-1}$, we have $\mathrm{deg}(W)=2^{n-1}-n$. 
Because the defining equations of $X$ are of the same degree, the base locus of the linear systems (\ref{linear system}) is $X$. Thus $W$ is smooth away from $X$. By the proof of Proposition \ref{jacobian transversality} $W$ is smooth at $X \cap W$. Thus $W$ is smooth for general $Z$.
Applying (\ref{eq:arithmetic genus}) again, we get $g_{W}=(n-3)(2^{n-2}-n)$. Therefore,
\[W=\begin{cases}
     \Pj_\K^1 & \text{if } n=3 \\
    \text{the rational normal curve in}\ \Pj_\K^4 & \text{if}\
n=4.
\end{cases}\]
For $n \geq 5$, we have $g_{W} \geq 6$.
\end{ex}

\subsection{Some remarks about computability}\label{computability}
The construction of $Z$ can be implemented with a computer algebra package. The
choices of general hyperplanes in a suitable projective space correspond to choosing
$\ell_i,i=1,\ldots,n-1$, general members of certain linear systems. We also choose a
general hyperplane at infinity $\mathbb{V}(h)$ that avoids $X\cap W$. To
verify that such general choices give us the $Z$ and $\tilde{Z}$ with the prescribed
properties we need to perform four tests.
\begin{enumerate}
\item First, we need to check that $(\Id_{Z},h)$ defines a zero-dimensional scheme
for some hyperplane $\mathbb{V}(h)$. This will ensure that $Z$ is a complete
intersection.
\item Second, we need to check that $\mathbb{V}(\Jac(Z),I_X,\ell_{n-1})=\emptyset$, where
$\Jac(Z)$ is the Jacobian ideal of $Z$. Equivalently, the condition is verified if the saturation
of the ideal $(\Jac(Z),I_X,\ell_{n-1})$ by the irrelevant ideal $\m=(x_0,\ldots,x_n)$ is $(1)$.
Passing this test will imply that $Z$ is generically reduced along $X$ and that $\ell_{n-1}$ satisfies the
property that $\mathbb{V}(\ell_{n-1})$ does not contain points on $X$ where the rank of the
Jacobian matrix $J(Z)$ drops.
\item Third, to check that $\tilde{Z}$ is a reduced complete intersection, it suffices to
verify that $\mathbb{V}(\Jac(\tilde{Z}),I_{\tilde{Z}})$ is a zero-dimensional scheme.
\item Finally, we need to verify that $\mathbb{V}(h)\cap X\cap W=\emptyset$. This
can be carried out by computing the $\K$-vector space dimension of the homogeneous
ideal quotient $S(\Pj^n_\K)/(\Jac(Z),I_X,h)$. This dimension is finite if and only
if the ideal is $\m$-primary. Thus $\mathbb{V}(h)\cap X\cap W=\emptyset$ if and only if
$\dim_\K S(\Pj^n_\K)/(\Jac(Z),I_X,h)<\infty$.
\end{enumerate}
The verification that an ideal defines a zero-dimensional scheme or an empty set can
be carried out in the standard affine charts using for example the $\verb|\vdim|$
operation in \textsc{Singular} \cite{DGPS}. It is proved that algorithms for
computing the number of solutions of zero-dimensional systems of equations have good complexity, i.e.\ they are polynomial in $d^{n+1}$ where $d$  is the maximum degree of input polynomials \cite{HL}. Once these tests are passed, we can compute $I(X,W)$ when $X$ is smooth
from the above data directly with another $\verb|\vdim|$ computation
\[I(X,W)=\dim_\K S(\Pj^n_\K)/(\Jac(Z),I_X,h-1).\]
Our approach does not require the heavy computations of primary decompositions, which are required to determine the ideal of $W$. The above computations generalize when $X$ has locally smoothable singularities provided that the equations of the smoothings of the singularities of $X$ are known.

\newpage

\appendix

\section*{Appendix}
\begin{center}
    {\footnotesize\textsc{Marc Chardin}\footnote{\scshape
    Institut Mathématique de Jussieu, CNRS and Sorbonne Université, 4 place Jussieu, E-mail:  marc.chardin@imj-prg.fr
    }}
\end{center}
\renewcommand{\thesubsection}{\arabic{subsection}}
\setcounter{equation}{0}

\subsection{A formula for the genus of geometrically linked curves}

Below we prove (\ref{eq:arithmetic genus}) algebraically with no assumptions on the field $\K$. Let $Z$ be a curve in $\Pj^n_\K$ defined by $n-1$ homogeneous equations of degrees
$d_1,\ldots ,d_{n-1}$. Then the quotient ring $R/I_Z$ is Gorenstein of dimension two; its Hilbert series is the
power series expansion of the fraction
\begin{align*}
S_Z (t)
&:=\frac{\prod_{i=1}^{n-1}(1-t^{d_i})}{(1-t)^{n+1}}
=\frac{\pi}{(1-t)^2}-\frac{\sigma \pi}{2(1-t)}+Q(t)
\end{align*}
with $\pi:=d_1\cdots d_{n-1}$, $\sigma :=\sum_{i=1}^{n-1}(d_i-1)$ and $Q$ a polynomial of
degree $\sigma-2$. Writing the Hilbert polynomial of a closed subscheme $Y$ of $\Pj^n_\K$  of dimension one as
\[
P_Y (\mu )=d_Y (\mu +1)-e_Y,
\]
it shows that $d_Z=\pi$ and $e_Z=\frac{1}{2}\sigma\pi$. Recall that the arithmetic genus of $Y$ is $$p_a (Y):=1-P_Y (0)=1+e_Y-d_Y.$$



Assume $I_Z =I\cap J$ with $X:=\proj (R/I)$ and $W:=\proj (R/J)$ of pure dimension $1$
and $I+J$ of codimension at least $n$. 
Equivalently, $I=\bigcap_{i\in E}\q_{i}$ and $J=\bigcap_{i\in F}\q_{i}$ with
$I_Z=\bigcap_{i\in E\cup F}\q_i$ the minimal primary decomposition of $I_Z$ and
$E\cap F=\emptyset$. In terms of liaison, we say that $X$ and $W$ are {\it geometrically linked} by $Z$.
Then the exact sequence
\[
0\rightarrow R/(I\cap J)\rightarrow R/I\oplus R/J\rightarrow R/(I+J)\rightarrow 0
\]
implies that $d_X+d_W=d_Z=\pi$ and $-e_X-e_W=-e_Z+I(X,W)$, with $X\cap W:=\proj (R/(I+J))$,
since it shows that $P_X+P_W=P_Z+P_{X\cap W}$ and $P_{X\cap W}$ is a constant equal to $I(X,W)$.
In this setting, $W\cap X$ is non-empty and locally Gorenstein, and $R/(I+J)$ is Gorenstein
if further $R/I$ is Cohen-Macaulay (equivalently if $R/J$ is so), by
\cite[Proposition 1.3 and Remarque 1.4]{PS}.\medskip

Now, $ I/I_Z=(I_Z:J)/I_Z\simeq\Hom_R (R/J,R/I_Z)
$, where the isomorphism sends an element $x$ to $(1\mapsto x)$
and $R/I_Z$ is Gorenstein, with canonical module
$$\omega_{R/I_Z}=\Ext^{n-1}_R (R/I_Z, R[-n-1])\simeq R/I_Z (\sigma -2),$$
as it follows by dualizing the Koszul complex on the generators of $I_Z$ into $R[-n-1]$,
since $\sigma-2 = d_1+\cdots+d_{n-1}-n-1$. 

Therefore, $ I/I_Z\simeq \Hom_R (R/J,\omega_{R/I_Z}(2-\sigma ))\simeq \omega_{R/J}(-\sigma +2)$ and the exact sequence $0\rightarrow I/I_Z\rightarrow R/I_Z\rightarrow R/I\rightarrow 0$ then shows that:

\begin{equation}\label{HP}
    P_Z(\mu )=P_X (\mu )+P_{\omega_W}(\mu -\sigma +2).
\end{equation}

To compute $P_{\omega_W}$, the simplest way is to use Serre duality: for every $\mu\in\Z$, 
\begin{align*}
P_W(\mu)
&=h^0(W,\Or_W (\mu ))-h^1(W,\Or_W(\mu ))\\
&=h^1(W,\omega_W (-\mu ))-h^0(W,\omega_W (-\mu ))\\
&=-P_{\omega_W}(-\mu ).\\
\end{align*}
Alternatively, if $F_\bullet$ is a finite graded resolution of $R/J$, then
$D_\bullet :=\Hom_R (F_\bullet ,R[-n-1])$ has a unique homology module
in high enough degrees, $\omega_{R/J}\simeq \Ext^{n-1}_R(R/J,R[-n-1])$.
Indeed, since $R/J$ is Cohen-Macaulay off the graded maximal ideal of $R$,
$\Ext^{i}_R(R/J,R)$  has finite length for $i\not=n-1$, hence is concentrated
in finitely many degrees. 

The resolution provides the Hilbert polynomial of $\omega_{R/J}$ in terms of the Hilbert polynomial of $R/J$,
since the alternated sum of Hilbert series $\sum_i (-1)^i S_{F_i}$ is $S_{R/J}$,
while $\sum_i (-1)^i S_{D_i}$ equals the alternated sum of the Hilbert series of the
graded modules $\Ext^{i}_R(R/J,R[-n-1])$, and can be computed from it. 

More precisely, writing $S_{R/J} (t)=\sum_i (-1)^i \frac{{P_i}(t)}{(1-t)^{n+1}}$
with $P_i (t)=\sum_j t^{b_{i,j}}\in \Z [t^{-1},t ]$ deduced from the expression
$F_i=\oplus_j R[-b_{i,j}]$, it follows that
\begin{align*}
\sum_i(-1)^iS_{D_i}(t)
&=t^{n+1}\sum_i (-1)^i \frac{{P_i}(t^{-1})}{(1-t)^{n+1}}\\
&=(-1)^{n+1}\sum_i (-1)^i \frac{{P_i}(t^{-1})}{(1-t^{-1})^{n+1}}\\
&=(-1)^{n+1}S_{R/J}(t^{-1}).
\end{align*}

Now, if $P(\mu)$ is the Hilbert polynomial associated to the series $S(t)$,
then the Hilbert polynomial associated to $S(t^{-1})$, rewritten as a series in $t$,
is $-P(-\mu )$.  This is \cite[Remark 1.8]{CEU}; it corresponds here to the identities 
\[
S_{R/J}(t^{-1})
=\frac{d_W}{(1-t^{-1})^2} - \frac{e_W}{(1-t^{-1})} + Q(t^{-1})
=\frac{d_W}{(1-t)^2} - \frac{2d_W-e_W}{(1-t)} + d_W - e_W + Q(t^{-1})
\]
with $Q$ a polynomial and $d_W(\mu +1)-(2d_W-e_W)=-[(-\mu +1)d_W-e_W]=-P_W (-\mu )$.
It follows that 
$P_{\omega_{R/J}}(\mu )=(-1)^{(n+1)-(n-1)}\times -P_{W} (-\mu )= -P_{W} (-\mu )$. \\



With this relation between $P_W$ and $P_{\omega_W}$, (\ref{HP}) yields:

\[
\begin{array}{rl}
P_X(\mu )&=P_Z(\mu )-P_{\omega_W}(\mu -\sigma +2)\\
&=\pi (\mu +1)-e_Z+[(-(\mu -\sigma +2)+1)d_W-e_W]\\
&=(\pi -d_W)(\mu +1)-(e_Z+e_W-\sigma d_W)\\
&=d_X (\mu+1)-(2e_Z-e_X-I(X,W)-\sigma (\pi -d_X))\\
&\phantom{d_X (\mu+1)d_X (\mu+1)d_X (\mu+1)}(e_W=e_Z -e_X-I(X,W),\ d_W=\pi -d_X )\\
&=d_X (\mu +1)-(\sigma d_X-e_X-I(X,W))\quad\text{since}\quad 2e_Z=\sigma \pi.\\
\end{array}
\]

\medskip

This gives us $e_X=\sigma d_X-e_X-I(X,W)$; hence $2e_X=\sigma d_X- I(X,W)$ and

\[
\begin{array}{rl}
p_a (X)&=1-P_X (0)\\
&=1-d_X+e_X\\
&=1-d_X+\frac{1}{2}(\sigma d_X-I(X,W))\\
&=1+\frac{1}{2}((\sigma -2) d_X-I(X,W)).\\
\end{array}
\]

\subsection{On the singularities of general links.}\medskip

Singularities of generic or general links have been studied in several situations, notably in the local case, see \cite[Proposition 2.9]{HU}. 

Due to our interest in applications of estimates for Castelnuovo--Mumford regularity,
Bernd Ulrich and the author provided versions of these results in a graded setting,
together with some additions on the nature of singularities and an extension to
residual intersections; our main result is \cite[Theorem 4.4]{CU} which implies the following result. 

\begin{thprop}
Let $X$ be a geometrically reduced local complete intersection in $\Pj^n_\K$ with $\K$ an infinite field. Suppose $X$ is of pure dimension $d \leq 3$ and that $X$ has isolated singularities. If $X$ is defined by equations of degrees
$d_1\geq\cdots\geq d_{r}$, there exists a complete intersection $Z=X\cup W$ such that $X$ and $W$ are geometrically linked by $Z$, the defining equations of $Z$ in $\Pj^n_\K$ are of degrees $d_1,\ldots ,d_{n-d}$, and such that $W$
and $X\cap W$ are smooth.
\end{thprop}

In the case of curves, this result implies that $W$ is smooth and locally at points of $X\cap W$,  $X$ and $W$ are smooth with distinct tangent lines, which is likely stated in other sources.

\end{document}